\documentclass[11pt,letterpaper]{amsart}

\usepackage{amsmath}
\usepackage{amssymb}
\usepackage{amsthm}
\usepackage[msc-links,abbrev]{amsrefs}
\usepackage{hyperref}
\usepackage[noabbrev,capitalize]{cleveref}
\usepackage{mathrsfs}
\usepackage{mathtools}
\usepackage{slashed}
\usepackage{tikz-cd}
\usepackage{enumitem}
\usepackage[noadjust]{cite}

\setenumerate{label=(\roman*)}

\DeclareMathOperator{\tr}{tr}

\DeclareMathOperator{\dvol}{dV}

\DeclareMathOperator{\Ric}{Ric}

\DeclareMathOperator{\Rm}{Rm}

\newcommand{\cg}{\widetilde{g}}

\newcommand{\cf}{\widetilde{f}}

\newcommand{\cu}{\widetilde{u}}
\newcommand{\cv}{\widetilde{v}}

\newcommand{\cD}{\widetilde{D}}

\newcommand{\cI}{\widetilde{I}}

\newcommand{\cR}{\widetilde{R}}

\newcommand{\cnabla}{\widetilde{\nabla}}

\newcommand{\cDelta}{\widetilde{\Delta}}

\newcommand{\cmE}{\widetilde{\mathcal{E}}}

\newcommand{\cmG}{\widetilde{\mathcal{G}}}

\DeclareMathOperator{\cRm}{\widetilde{\Rm}}

\newcommand{\mC}{\mathcal{C}}
\newcommand{\mD}{\mathcal{D}}
\newcommand{\mE}{\mathcal{E}}

\newcommand{\mG}{\mathcal{G}}

\newcommand{\mI}{\mathcal{I}}

\newcommand{\mM}{\mathcal{M}}

\newcommand{\mP}{\mathcal{P}}

\newcommand{\mR}{\mathcal{R}}
\newcommand{\mS}{\mathcal{S}}

\newcommand{\kD}{\mathfrak{D}}

\newcommand{\kc}{\mathfrak{c}}

\newcommand{\bN}{\mathbb{N}}

\newcommand{\bR}{\mathbb{R}}

\newcommand{\RHS}{\operatorname{RHS}}
\newcommand{\LHS}{\operatorname{LHS}}

\def\sideremark#1{\ifvmode\leavevmode\fi\vadjust{\vbox to0pt{\vss
 \hbox to 0pt{\hskip\hsize\hskip1em
 \vbox{\hsize3cm\tiny\raggedright\pretolerance10000
 \noindent #1\hfill}\hss}\vbox to8pt{\vfil}\vss}}}

\newcommand{\suchthat}{\mathrel{}:\mathrel{}}

\newcommand{\ssfu}{{\ast}}
\newcommand{\ssfv}{{\diamond}}

\newtheorem{theorem}{Theorem}[section]
\newtheorem{proposition}[theorem]{Proposition}
\newtheorem{lemma}[theorem]{Lemma}
\newtheorem{corollary}[theorem]{Corollary}

\theoremstyle{definition}

\numberwithin{equation}{section}

\allowdisplaybreaks[4]

\begin{document}

\title{Juhl type formulas for curved Ovsienko--Redou operators}

\author[S. Chern]{Shane Chern}
\address[S. Chern]{Fakult\"at f\"ur Mathematik \\ Universit\"at Wien \\ Vienna \\ A-1090 \\ Austria}
\email{chenxiaohang92@gmail.com}

\author[Z. Yan]{Zetian Yan}
\address[Z. Yan]{Department of Mathematics \\ UC Santa Barbara \\ Santa Barbara \\ CA 93106 \\ USA}
\email{ztyan@ucsb.edu}

\begin{abstract}
We prove Juhl type formulas for the curved Ovsienko--Redou operators and their linear analogues, which indicate the associated formal self-adjointness, thereby confirming two conjectures of Case, Lin, and Yuan. We also offer an extension of Juhl's original formula for the GJMS operators.
\end{abstract}

\maketitle

\section{Introduction}\label{introduction}

The \emph{GJMS operator} of order $2k$ is a conformally invariant differential operator with leading-order term $\Delta^k$ defined on any Riemannian manifold $(M^n,g)$ of dimension $n \geq 2k$, and this family generalizes the well-known second-order conformal Laplacian (also called Yamabe operator) and the fourth-order operator discovered by Paneitz~\cite{Paneitz1983}. The GJMS operators have been studied intensively during the past decades in connection with, for example, prescribed $Q$-curvature problems, higher-order Sobolev trace inequalities, scattering theory on conformally compact manifolds, and functional determinant quotient formulas for pairs of metrics in a conformal class.

Based on a theory of residue families, in a series of works~\cite{Juhl2009,Juhl2013}, Juhl derived remarkable formulas that express GJMS operators as a sum of compositions of lower-order GJMS operators up to a certain second-order term or as linear combinations of compositions of second-order differential operators, through an ingenious inversion relation for compositions given credit to Krattenthaler~\cite[Theorem 2.1]{Juhl2013}. Later, Fefferman and Graham~\cite{FeffermanGraham2013} provided an alternative proof of Juhl's formulas, starting directly from the original construction on the ambient space but also requiring Krattenthaler's insight. Juhl's formulas have significant applications in the aforementioned study of GJMS operators, such as the asymptotic expansion of the heat kernel~\cite{Juhl2016} and prescribed higher-order $Q$-curvature problems~\cite{MazumdarVetois2020}.



To study conformally covariant operators of rank three, Case, Lin, and Yuan~\cite{CaseLinYuan2022or} gave two generalizations of the GJMS operators;
here we focus on those that are formally self-adjoint.
The first is a family of conformally invariant bidifferential operators:
\begin{equation*}
	D_{2k} \colon \mE\left[ -\frac{n-2k}{3} \right]^{\otimes 2} \to \mE\left[-\frac{2n+2k}{3}\right]
\end{equation*}
of total order $2k$.
They are called the \emph{curved Ovsienko--Redou operators} because they generalize a family of bidifferential operators constructed by Ovsienko and Redou~\cite{OvsienkoRedou2003} on the sphere. Let
\begin{align*}
	a_{r,s,t} := \binom{k}{r,s,t} \frac{\Gamma\bigl(\frac{n+4k}{6}-r\bigr)\Gamma\bigl(\frac{n+4k}{6}-s\bigr)\Gamma\bigl(\frac{n+4k}{6}-t\bigr)}{\Gamma\bigl(\frac{n-2k}{6}\bigr)\Gamma\bigl(\frac{n+4k}{6}\bigr)^2}
\end{align*}
with $\binom{k}{r,s,t}=\frac{k!}{r!s!t!}$ the multinomial coefficient wherein $k=r+s+t$. The operators $D_{2k}$ are determined ambiently by
\begin{align*}
	\cD_{2k}(\cu \otimes \cv)  := \sum_{r+s+t=k} a_{r,s,t} \cDelta^r\left( (\cDelta^s\cu)(\cDelta^t\cv) \right)
\end{align*}
on $\cmE\bigl[-\frac{n-2k}{3}\bigr] \otimes \cmE\bigl[-\frac{n-2k}{3}\bigr]$;
in this paper, tensor products are over $\bR$. The second generalization is a family of conformally invariant differential operators:
\begin{equation*}
	D_{2k;\mI} \colon \mE\left[ -\frac{n-2k-2\ell}{2} \right] \to \mE\left[-\frac{n+2k+2\ell}{2} \right]
\end{equation*}
of order $2k$ associated with a scalar Weyl invariant $\mI$ of weight $-2\ell$. These are determined ambiently by a scalar Riemannian invariant $\cI$ of weight $-2\ell$ as
\begin{align*}
	\cD_{2k;\cI}(\cu) & := \sum_{r+s=k} b_{r,s}\cDelta^r\left( \cI \cDelta^s\cu \right)
\end{align*}
on $\cmE\bigl[-\frac{n-2k-2\ell}{2}\bigr]$ where
\begin{align*}
	b_{r,s} := \binom{k}{s}\frac{\Gamma(\ell+s)\Gamma(\ell+r)}{\Gamma(\ell)^2}.
\end{align*}
We refer the reader to \cref{sec:bg} for an explanation of our notation and a description of how the ambient formulas determine conformally invariant operators.

In a recent work, Case and the second named author~\cite{CaseYan2024} proved the formal self-adjointness of the two families of operators, thereby answering in the affirmative two conjectures of Case, Lin, and Yuan~\cite{CaseLinYuan2022or}. Taking $D_{2k}$ as an example, their main idea is to factorize $\cD_{2k}$ on the Poincar\'e space and realize the Dirichlet form of $D_{2k}$ as the coefficient of the logarithmic term in the Dirichlet form of $\cD_{2k}$. The \emph{Divergence Theorem} then yields the desired formal self-adjointness. This method avoids the lack of either an equivalent description of these operators as an obstruction to solving some second-order PDE, on which the Graham--Zworski argument~\cite{GrahamZworski2003} highly relies, or the complicated combinatorial arguments required by Juhl in~\cite{Juhl2013}. 

However, the explicit structures of these two families of operators are not revealed through the arguments in \cite{CaseYan2024}, especially in view of Juhl's formulas for the GJMS operators \cite{Juhl2009,Juhl2013}. To understand the underlying formal self-adjointness for the operators $D_{2k}$ and $D_{2k;\mI}$ in a more direct way and to read the two operators from a more general setting, the main purpose of this paper is to prove the following two Juhl type formulas:

\begin{theorem}\label{JUhl-formula-ovsienko-redou}
	Let $(M^n,\kc)$ be a conformal manifold.
	Let $k \in \bN$ and if $n$ is even, we assume additionally that $k \leq \frac{n}{2}$. Writing $L_k:=\frac{n}{6}+\frac{2k}{3}$, the operator $D_{2k}$ satisfies
	\begin{align}
		D_{2k}(u\otimes v) & = \sum_{\mathbf{A}'}\sum_{\mathbf{A}^{\ssfu}}\sum_{\mathbf{A}^{\ssfv}} \mM_{2(A'_{r'}+1)}\cdots \mM_{2(A'_1+1)}\\
		&\ \quad \big(\mM_{2(A_{r^{\ssfu}}^{\ssfu}+1)}\cdots \mM_{2(A_1^{\ssfu}+1)}(u)\mM_{2(A_{r^{\ssfv}}^{\ssfv}+1)}\cdots \mM_{2(A_1^{\ssfv}+1)}(v)\big)\notag\\
		&\quad\times \frac{k!}{L_k^2}\prod_{n=1}^k (L_k-n)\prod_{i=1}^{r^{\ssfu}} \frac{1}{(A_{i}^{\ssfu}!)^2} \prod_{i=1}^{r^{\ssfv}} \frac{1}{(A_{i}^{\ssfv}!)^2} \prod_{i=1}^{r'} \frac{1}{(A'_{i}!)^2}\notag\\
		&\quad\times \prod_{i=1}^{r^\ssfu} \scalebox{0.75}{$\dfrac{1}{\sum\limits_{j=1}^i (A_{j}^\ssfu+1)}$}\prod_{i=1}^{r^\ssfu} \scalebox{0.75}{$\dfrac{1}{L_k-\sum\limits_{j=1}^i (A_{j}^\ssfu+1)}$} \prod_{i=1}^{r^\ssfv} \scalebox{0.75}{$\dfrac{1}{\sum\limits_{j=1}^i (A_{j}^\ssfv+1)}$}\prod_{i=1}^{r^\ssfv} \scalebox{0.75}{$\dfrac{1}{L_k-\sum\limits_{j=1}^i (A_{j}^\ssfv+1)}$}\notag\\
		&\quad\times \prod_{i=1}^{r'} \scalebox{0.75}{$\dfrac{1}{\sum\limits_{j=1}^i (A'_{r'+1-j}+1)}$} \prod_{i=1}^{r'} \scalebox{0.75}{$\dfrac{1}{L_k-\sum\limits_{j=1}^i (A'_{r'+1-j}+1)}$},\notag
	\end{align}
	where the summation runs over all nonnegative sequences $\mathbf{A}^{\ssfu}=(A_1^{\ssfu},\ldots,A_{r^{\ssfu}}^{\ssfu})$, $\mathbf{A}^{\ssfv}=(A_1^{\ssfv},\ldots,A_{r^{\ssfv}}^{\ssfv})$ and $\mathbf{A}'=(A'_1,\ldots,A'_{r'})$ such that
	\begin{align*}
		r^{\ssfu} + \sum\limits_{j=1}^{r^{\ssfu}} A_j^{\ssfu} + r^{\ssfv} + \sum\limits_{j=1}^{r^{\ssfv}} A_j^{\ssfv} + r' + \sum_{j=1}^{r'} A'_j = k.
	\end{align*}
	In particular, $D_{2k}$ is formally self-adjoint.
\end{theorem}

\begin{theorem}\label{JUhl-formula-linear}
	Let $(M^n,\kc)$ be a conformal manifold and let $\cI \in \cmE[-2\ell]$ be an ambient scalar Riemannian invariant. Let $k \in \bN$ and if $n$ is even, we assume additionally that $k+\ell\le \frac{n}{2}+1-\delta_{0,\ell}$ with $\delta$ the Kronecker delta. Then the operator $D_{2k;\mI}$ satisfies
	\begin{align}
		D_{2k;\mI}(u)& = \sum_{R}\sum_{\mathbf{A},\mathbf{A}'}\mM_{2(A'_{r'}+1)}\cdots \mM_{2(A'_1+1)} \big(\cI^{(R)}\mM_{2(A_r+1)}\cdots \mM_{2(A_1+1)}(u)\big)\\
		&\quad\times (-1)^{R}\, 2^{R}\, \frac{k!}{R!}\, \prod_{n=0}^{k-1} (\ell+n)^2 \prod_{i=1}^r \frac{1}{(A_{i}!)^2} \prod_{i=1}^{r'} \frac{1}{(A'_{i}!)^2}\notag\\
		&\quad\times \prod_{i=1}^{r} \scalebox{0.75}{$\dfrac{1}{\sum\limits_{j=1}^i (A_{j}+1)}$}\prod_{i=1}^{r} \scalebox{0.75}{$\dfrac{1}{\ell+k-\sum\limits_{j=1}^i (A_{j}+1)}$}\notag\\
		&\quad\times \prod_{i=1}^{r'} \scalebox{0.75}{$\dfrac{1}{\sum\limits_{j=1}^i (A'_{r'+1-j}+1)}$}\prod_{i=1}^{r'} \scalebox{0.75}{$\dfrac{1}{\ell+k-\sum\limits_{j=1}^i (A'_{r'+1-j}+1)}$},\notag
	\end{align}
	where the summation runs over all nonnegative integers $R$ and all sequences $\mathbf{A}=(A_{1},\ldots,A_{r})$ and $\mathbf{A}'=(A'_1,\ldots,A'_{r'})$ of nonnegative integers such that
	\begin{align*}
		R + r + \sum\limits_{j=1}^r A_j + r' + \sum_{j=1}^{r'} A'_j  = k.
	\end{align*}
	In particular, $D_{2k;\mI}$ is formally self-adjoint.
\end{theorem}

In the above theorems, the additional condition for the even $n$ case is imposed to ensure that the operators $D_{2k}$ and $D_{2k;\mI}$, respectively, are independent of the ambiguity of the ambient metric. This will be discussed in Section~\ref{sec:bg}.

\subsection{Outline of the idea}

\subsubsection{Fefferman and Graham's argument}

In~\cite{FeffermanGraham2013}, Fefferman and Graham gave a direct proof of Juhl's formula for the GJMS operators, starting from the original construction on the ambient space. In this subsection, we sketch their main idea and show how to improve their arguments to work with the differential operators $D_{2k}$ and $D_{2k;\mI}$.

Given a Riemannian manifold $(M^n,g)$, let $(\cmG,\cg)$ be its straight and normal ambient space defined in \eqref{eqn:straight-and-normal}. We set 
\begin{equation*}
	w(\rho):=\left(\frac{\det g_{\rho}}{\det g}\right)^{\frac{1}{4}},
\end{equation*}
and denote $\cDelta_{w}:=w\circ \cDelta_{g_{\rho}} \circ w^{-1}$. Since the construction of GJMS operators $P_{2k}$ is independent of the extension of $u$ on $M^n\times (-\epsilon, \epsilon)$, Fefferman and Graham reformulated them via $\cDelta_{w}$ instead of $\cDelta_{g_{\rho}}$. Through a direct computation~\cite[eq.~(2.4)]{FeffermanGraham2013}, one has, for $u=u(x,\rho)$,
\begin{equation}\label{Laplacian-expansion}
	\cDelta_w(\tau^{\gamma}u)=\tau^{\gamma-2}\left[-2\rho \partial^2_{\rho}+(2\gamma+n-2)\partial_{\rho}+\widetilde{\mM}(\rho)\right]u.
\end{equation}
It is notable that $\widetilde{\mM}(\rho)$ is a second-order, formally self-adjoint operator on $(M^n,g)$ for each $\rho\in \bR$ and we may regard it as the generating function
\begin{align*}
	\widetilde{\mathcal{M}}(\rho) := \sum_{N\ge 0} \frac{1}{\big(N!\big)^2}\left(-\frac{\rho}{2}\right)^N \mM_{2(N+1)}
\end{align*}
for a family of second-order, formally self-adjoint operators $\{\mM_{2(N+1)}\}_{N\in \bN}$ on $(M^n,g)$. Now setting
\begin{align*}
	\mathcal{R}_j := -2\rho\partial_\rho^2 + 2j \partial_\rho + \widetilde{\mathcal{M}}(\rho),
\end{align*}
we may formulate the GJMS operators $P_{2k}$ as
\begin{align}
	P_{2k} (u) := \mR_{1-k}\mR_{3-k}\cdots \mR_{k-3}\mR_{k-1} (u)\Big\vert_{\rho=0}.
\end{align}
With recourse to a nice combinatorial argument, Fefferman and Graham proved the following result~\cite[eq.~(2.5)]{FeffermanGraham2013}:

\begin{theorem}
	\label{Juhl-formula}
	For $k\geq 1$ and additionally $k\leq \frac{n}{2}$ if $n$ is even, 
	\begin{align}\label{eq:FG-formula}
		P_{2k}(u) &= \sum_{\mathbf{A}}  \mM_{2(A_r+1)}\cdots \mM_{2(A_1+1)} (u)\\
		&\quad \times \big((k-1)!\big)^2 \prod_{i=1}^{r}\frac{1}{\big(A_i!\big)^2} \prod_{i=1}^{r-1} \scalebox{0.75}{$\dfrac{1}{\sum\limits_{j=1}^{i} (A_{j}+1)}$} \prod_{i=1}^{r-1} \scalebox{0.75}{$\dfrac{1}{\sum\limits_{j=1}^i (A_{r+1-j}+1)}$},\notag
	\end{align}
	where the summation runs over all sequences $\mathbf{A}=(A_{1},\ldots,A_{r})$ of nonnegative integers such that
	\begin{align*}
		\sum_{i=1}^r (A_{i}+1) = k.
	\end{align*}
	In particular, $P_{2k}$ is formally self-adjoint.
\end{theorem}

Here the main problem is to compute the coefficient for each composition of $\mM$-operators on the right-hand side of \eqref{eq:FG-formula}. To do so, we fix the $\mM_{2(A_r+1)}\cdots \mM_{2(A_1+1)}$ to be worked with and start with a truncated composition $\mR_{k+1-2i}\cdots \mR_{k-3}\mR_{k-1}$ with $1\leq i\leq k$. To produce the desired $\mM$-composition, we shall look at terms containing a $\mM$-truncation $\mM_{2(A_j+1)}\cdots \mM_{2(A_1+1)}$ with $1\le j\le r$ in the expansion of our selected truncated composition of $\mR$-operators. These terms are multiples of the $\mM$-truncation times a power of $\rho$. Note that amongst the $\mR$-operators in the truncated composition, it is clear that the zeroth term $\widetilde{\mM}(\rho)$ has contributed exactly $j$ times and the differentiation on $\rho$ has contributed $i-j$ times, thereby implying that the aforementioned power of $\rho$ has an exponent $\sum_{m=1}^j (A_{m}+1)-i$. Finally, to compute the coefficients associated with the $\mM$-truncations in question, Fefferman and Graham cleverly showed that these coefficients satisfy a family of recursive relations~\cite[eq.~(3.5)]{FeffermanGraham2013} by a remarkable combinatorial argument, and hence arrived at \eqref{eq:FG-formula} by solving these recursions.


When it comes to the curved Ovsienko--Redou operators $D_{2k}$, it can be shown that
\begin{align}\label{eq:D-2k}
	&D_{2k}(u \otimes v)\\
	&\qquad= \sum_{r+s+t=k} a_{r,s,t} \mR_{-L_k+1}\cdots \mR_{-L_k+2r-3}\mR_{-L_k+2r-1}\notag\\
	&\qquad\quad \big( \mR_{L_k+1-2s}\cdots \mR_{L_k-3}\mR_{L_k-1} (u)\mR_{L_k+1-2t}\cdots \mR_{L_k-3}\mR_{L_k-1} (v)\big)\Big\vert_{\rho=0},\notag
\end{align}
where we set $L_k:=\frac{n}{6}+\frac{2k}{3}$. However, this time we cannot proceed with the argument of Fefferman and Graham since an analogous family of recursions becomes out of reach, mainly due to the twisted inner layer of compositions. Likewise, for the operator $D_{2k;\mI}$, we have, as formulated in~\cite[eq.~(6.2)]{CaseLinYuan2022or},
\begin{align}\label{eq:D-2k-I}
	D_{2k;\mI}(u)&=\sum_{r+s=k} b_{r,s} \mR_{k-\ell-2s+1-2r} \cdots \mR_{k-\ell-2s-3}\mR_{k-\ell-2s-1}\\
	&\quad\big( \cI \mR_{k+\ell-2s+1}\cdots \mR_{k+\ell-3}\mR_{k+\ell-1} (u)\big)\Big\vert_{\rho=0}. \notag
\end{align}
The inserted ambient scalar Riemannian invariant $\cI$ also kills the expected recursive relations.

\subsubsection{Casting the ``Diffindo'' charm}\label{sec:diffindo-intro}

To overcome the issue caused by the lack of necessary recursions, we need to cast the charm of \emph{Diffindo}\footnote{This spell means making seams split open and severing an object into two pieces.} for the two operators $D_{2k}$ and $D_{2k;\mI}$. That is, we shall separate the analyses of the inner and outer layers of operator compositions.

We begin with the inner layer. Let $u\in \mC^{\infty}(M^n)$. For $M\in\bN$ and $L\in \bR$, we consider
\begin{align*}
	\cD_{M,L}(u) := \mR_{L+1-2M} \cdots \mR_{L-3} \mR_{L-1} (u).
\end{align*}
It is notable that $\cD_{M,L}$ reduces to the GJMS operator $P_{2k}$ by choosing $L=M=k$ and taking $\rho=0$. By the definition of the $\mR$-operators, we see that for each $N\in \bN$, the expansion of $\cD_{M,L}(\rho^N u)$ is a linear combination of a nonnegative power of $\rho$ times a composition of $\mM$-operators acted on $u$. The takeaway from our analysis is that by using an evaluation of a hypergeometric series (instead of looking for recursions), the coefficients in this linear combination can be explicitly expressed, as shown in Corollary \ref{coro:D-u}.

Next, we continue with the outer layer. Note that the essential contributions of the inner layer are of the form $\rho^N$ for $D_{2k}$ and $\rho^N \cI$ for $D_{2k;\mI}$ where $N\in \bN$. The analysis for the former can be copied from that for the inner layer, while the study of the latter, as given in Corollary \ref{D-f-u}, is much more complicated, relying on a trick for Lemma \ref{le:S(A,B)-f}.

To finalize our arguments, we shall look not only at the operators $D_{2k}$ and $D_{2k;\mI}$, but also at their generalizations with a few more free parameters added, as given in \eqref{eq:D-general} and \eqref{generalized-linear-operator}, respectively. The main advantage of these free parameters is that the application of induction becomes possible. By further utilizing the combinatorial identities shown in Appendix \ref{sec:appendix}, we finally arrive at the explicit expressions of the two families of generalized operators presented in Theorems \ref{th:D-2k-general} and \ref{generalized-Juhl-linear-formula}. In particular, as pointed out in Corollaries \ref{formally-self-adjointness-2} and \ref{formal-self-adjointness}, the nature of formal self-adjointness of the operators $D_{2k}$ and $D_{2k;\mI}$ is \emph{exclusive} among the two generic families, making our Juhl type formulas more meaningful.

\subsubsection{Notation and basic properties}

To facilitate our analysis, we split the $\mR$-operators by defining for $f=f(\rho)$:
\begin{align*}
	\mD_j(f) &:= \big({-\rho}\partial_\rho^2 + j \partial_\rho\big)(f),\\
	\mP_k(f) &:= \rho^k f,
\end{align*}
where $j\in \bR$ and $k\in \bN$.

Let $u\in \mC^{\infty}(M^n)$. For the inner layer, we have introduced the operators $\cD_{M,L}$ for $M\in\bN$ and $L\in \bR$:
\begin{align}
	\cD_{M,L}(u) &:= \mR_{L+1-2M} \cdots \mR_{L-3} \mR_{L-1} (u).
\end{align}
Meanwhile, we define
\begin{align}\label{eq:D-M}
	D_{M,L}(u):= \cD_{M,L}(u)\Big\vert_{\rho=0}.
\end{align}
If we expand $\cD_{M,L}$ in terms of the operators $\mD$ and $\mP$, then all its terms are of the form
\begin{align}\label{eq:R-reduced}
	\mR^*_{L+1-2M} \cdots \mR^*_{L-3} \mR^*_{L-1} (u),
\end{align}
where $\mR^*_{L+1-2j}$ takes either $2\mD_{L+1-2j}$ or $\frac{1}{(k!)^2}(-\frac{1}{2})^k \mM_{2(k+1)}\mP_k$.
Hence, we may record the terms in the expansion of $\cD_{M,L}$ as
\begin{align}\label{eq:SAB-def}
	\mathcal{S}_{\mathbf{A},\mathbf{B}}(u) = \mathcal{S}_{\mathbf{A},\mathbf{B},M,L}(u) := \underbrace{\mD\cdots \mD}_{\text{$B_r$ times}} \mP_{A_\ell} \cdots \underbrace{\mD\cdots \mD}_{\text{$B_1$ times}} \mP_{A_1} \underbrace{\mD\cdots \mD}_{\text{$B_0$ times}}(u),
\end{align}
where $\mathbf{A}=(A_1,\ldots,A_r)$ and $\mathbf{B}=(B_0,B_1,\ldots,B_r)$ are sequences of \emph{nonnegative} integers with the length of $\mathbf{B}$ one more than that of $\mathbf{A}$ such that
\begin{align}\label{eq:B-cond}
	r + \sum_{j=0}^r B_j = M.
\end{align}
It is also notable that the omitted index of the $\mD$-operators should be determined by its position in $\mS_{\mathbf{A},\mathbf{B}}(u)$. Meanwhile, in the expansion of $\cD_{M,L}$, we need to attach to the term $\mS_{\mathbf{A},\mathbf{B}}(u)$ a coefficient:
\begin{align}\label{eq:S-prefactor}
	&2^{B_r} \cdot \frac{1}{(A_{r}!)^2}(-\tfrac{1}{2})^{A_{r}} \mM_{2(A_{r}+1)} \cdots 2^{B_1} \cdot \frac{1}{(A_{1}!)^2}(-\tfrac{1}{2})^{A_{1}} \mM_{2(A_{1}+1)}\cdot 2^{B_0} \\
	&\qquad\quad = (-1)^{\sum\limits_{j=1}^r A_j}\cdot 2^{-\sum\limits_{j=1}^r A_j + \sum\limits_{j=0}^r B_j}\cdot \prod_{i=1}^r \frac{1}{(A_{i}!)^2} \cdot \mM_{2(A_r+1)}\cdots \mM_{2(A_1+1)}.\notag
\end{align}

To study the curved Ovsienko--Redou operators $D_{2k}$, we look at a generic family of operators:
\begin{align}\label{eq:cD-triple}
	\cD_{[M',L'],[M^\ssfu,L^\ssfu],[M^\ssfv,L^\ssfv]}(u\otimes v) := \cD_{M',L'}\big(\cD_{M^\ssfu,L^\ssfu}(u)\cD_{M^\ssfv,L^\ssfv}(v)\big),
\end{align}
where $u,v\in \mC^{\infty}(M^n)$. Furthermore, we write
\begin{align}\label{eq:D-triple}
	D_{[M',L'],[M^\ssfu,L^\ssfu],[M^\ssfv,L^\ssfv]}(u\otimes v) := \cD_{M',L'}\big(\cD_{M^\ssfu,L^\ssfu}(u)\cD_{M^\ssfv,L^\ssfv}(v)\big)\Big\vert_{\rho=0}.
\end{align}
It is clear that the curved Ovsienko--Redou operator $D_{2k}$ is a specialization of
\begin{align}\label{eq:D-general}
	&D_{U,V,L,K^\ssfu,K^\ssfv}(u\otimes v)\\
	&\qquad:= \sum_{\substack{M^{\ssfu},M^{\ssfv},M'\ge 0\\ M^{\ssfu}+M^{\ssfv}+M'=U}} \frac{\Gamma(U+K^\ssfu+1)\Gamma(U+K^\ssfv+1)}{\Gamma(M^\ssfu+K^\ssfu+1)\Gamma(M^\ssfv+K^\ssfv+1)\Gamma(M'+1)}\notag\\
	&\qquad\ \quad\times \frac{\Gamma(L-M^\ssfu)\Gamma(L-M^\ssfv)\Gamma(L+V-M')}{\Gamma(L-U)\Gamma(L)^2}\notag\\
	&\qquad\ \quad\times D_{[M',-L-V+2M'],[M^{\ssfu},L-K^{\ssfu}],[M^{\ssfv},L-K^{\ssfv}]}(u\otimes v),\notag
\end{align}
where $U$ and $V$ are fixed \emph{nonnegative} integers and $L$, $K^\ssfu$ and $K^\ssfv$ are indeterminates.

For the operators $D_{2k;\mI}$, we look at another generic family of operators:
\begin{align}
	\cD_{[M',L'],[M,L];f}(u) &:= \cD_{M',L'}\big(f\,\cD_{M,L}(u)\big),
\end{align}
where $u\in \mC^{\infty}(M^n)$ and $f\in \mC^{\infty}(M^n\times (-\epsilon, \epsilon))$. Also, we write
\begin{align}\label{eq:L-def}
	D_{[M',L'],[M,L];f}(u) :=  \cD_{[M',L'],[M,L];f}(u)\Big\vert_{\rho=0}.
\end{align}
Now $D_{2k;\mI}$ can be generalized as
\begin{align}\label{generalized-linear-operator}
	D_{U,V,L,K;f}(u) &
	:= \sum_{\substack{M,M'\ge 0\\M+M'=U}} \frac{\Gamma(U+K+1)}{\Gamma(M+K+1)\Gamma(M'+1)}\\
	&\ \quad\times \frac{\Gamma(L+M')\Gamma(L+V-M')}{\Gamma(L)^2}\notag\\
	&\ \quad\times D_{[M',-L-V+2M'],[M,L-K+U];f}(u),\notag
\end{align}
where $U$ and $V$ are fixed \emph{nonnegative} integers and $L$ and $K$ are indeterminates.

\section{Background}
\label{sec:bg}

\subsection{Ambient spaces}
\label{subsec:ambient}

We begin by recalling the relevant aspects of the ambient space, following Fefferman and Graham~\cite{FeffermanGraham2012}.

Let $(M^n,\kc)$ be a conformal manifold of signature $(p,q)$.
Denote
\begin{equation*}
	\mG := \left\{ (x,g_x) \suchthat x \in M, g \in \kc \right\} \subset S^2T^\ast M 
\end{equation*}
and let $\pi \colon \mG \to M$ be the natural projection.
We regard $\mG$ as a principal $\bR_+$-bundle with dilation $\delta_\lambda \colon \mG \to \mG$ for $\lambda \in \bR_+$:
\begin{equation*}
	\delta_\lambda (x, g_x) := (x , \lambda^2 g_x) .
\end{equation*}
Denote by $T := \left. \frac{\partial}{\partial\lambda}\right|_{\lambda=1} \delta_\lambda$ the infinitesimal generator of $\delta_\lambda$.
The canonical metric is the degenerate metric ${\boldsymbol{g}}$ on $\mG$ defined by
\begin{equation*}
	{\boldsymbol{g}}(X,Y) := g_x(\pi_\ast X, \pi_\ast Y)
\end{equation*}
for $X,Y\in T_{(x,g_x)}\mG$.
Note that $\delta_\lambda^\ast \boldsymbol{g} = \lambda^2\boldsymbol{g}$.

A choice of representative $g \in \kc$ determines an identification $\bR_+ \times M \cong \mG$ via $(\tau,x)\cong (x,\tau^2g_x)$.
In these coordinates, $T_{(\tau,x)} = \tau\partial_{\tau}$ and $\boldsymbol{g}_{(\tau,x)} = \tau^2\pi^\ast g$.

Extend the projection and dilation to $\mG \times \bR$ in the natural way:
\begin{align*}
	\pi( x, g_x, \rho) & := x , \\
	\delta_\lambda( x, g_x, \rho) & := ( x , \lambda^2 g_x , \rho ) ,
\end{align*}
where $\rho$ denotes the coordinate on $\bR$.
We abuse notation and also denote by $T$ the infinitesimal generator of $\delta_\lambda \colon \mG \times \bR \to \mG \times \bR$.
Let $\iota \colon \mG \to \mG \times \bR$ denote the inclusion $\iota(x,g_x) := (x,g_x,0)$.
A \emph{pre-ambient space} for $(M^n,\kc)$ is a pair $(\cmG,\cg)$ consisting of a dilation-invariant subspace $\cmG \subseteq \mG \times \bR$ containing $\iota(\mG)$ and a pseudo-Riemannian metric $\cg$ of signature $(p+1,q+1)$ satisfying $\delta_\lambda^\ast \cg = \lambda^2 \cg$ and $\iota^\ast \cg = \boldsymbol{g}$.

An \emph{ambient space} for $(M^n,\kc)$ is a pre-ambient space $(\cmG,\cg)$ for $(M^n,\kc)$ which is formally Ricci flat. That is,
\begin{align*}
	\Ric(\cg) \in \begin{cases}
		O(\rho^\infty), & \text{if $n$ is odd},\\
		O^+(\rho^{n/2-1}), & \text{if $n$ is even}.
	\end{cases}
\end{align*}
Here $O^+(\rho^m)$ is the set of sections $S$ of $S^2T^\ast\cmG$ such that
\begin{enumerate}
	\item $\rho^{-m}S$ extends continuously to $\iota(\mG)$;
	\item for each $z = (x,g_x) \in \mG$, there is an $s \in S^2 T_{x}^\ast M$ such that $\tr_{g_x}s=0$ and $(\iota^\ast(\rho^{-m}S)(z) = (\pi^\ast s)(z)$.
\end{enumerate}

Fefferman and Graham~\cite[Theorem~2.9(A)]{FeffermanGraham2012} showed that:
{\itshape Letting $(M^n,\kc)$ be a conformal manifold and picking a representative $g \in \kc$, there is an $\epsilon>0$ and a one-parameter family $g_\rho$ of metrics on $M$ with $\rho \in (-\epsilon,\epsilon)$ such that $g_0=g$ and
	\begin{equation}
		\label{eqn:straight-and-normal}
		\begin{split}
			\cmG & := \mG \times (-\epsilon,\epsilon) \\
			\cg & := 2\rho \, d\tau^2 + 2\tau \, d\tau \, d\rho + \tau^2 g_\rho 
		\end{split}
	\end{equation}
	define an ambient space $(\cmG,\cg)$ for $(M^n,\kc)$.}
We say that an ambient metric in the above form is \emph{straight and normal}.

Let $(\cmG,\cg)$ be the ambient space for $(M^n,\kc)$.
Denote by
\begin{equation*}
	\cmE[w] := \left\{ \cf \in C^\infty(\cmG) \suchthat \delta_\lambda^\ast \cf = \lambda^w\cf \right\} 
\end{equation*}
the space of homogeneous functions on $\cmG$ of weight $w \in \bR$.
Note that $\cf \in \cmE[w]$ if and only if $T\cf = w\cf$.
The space of \emph{conformal densities} of weight $w$ is
\begin{equation*}
	\mE[w] := \left\{ \iota^\ast \cf \in C^\infty(\mG) \suchthat \cf \in \cmE[w] \right\} .
\end{equation*}

Fix $n \in \bN$.
An ambient \emph{scalar Riemannian invariant} $\cI$ is an assignment to each ambient space $(\cmG^{n+2},\cg)$ of a linear combination $\cI_{\cg}$ of complete contractions of
\begin{equation}
	\label{eqn:tensors}
	\cnabla^{N_1}\cRm \otimes \dotsm \otimes \cnabla^{N_\ell}\cRm 
\end{equation}
with $\ell \geq 2$, where $\cnabla$ and $\cRm$ are the Levi-Civita connection and Riemann curvature tensor, respectively, of $\cg$. We regard $\cRm$ as a section of $\otimes^4T^\ast\cmG$, and we use $\cg^{-1}$ to take contractions.
Any complete contraction of~\eqref{eqn:tensors} is homogeneous of weight
\begin{equation*}
	w = -2\ell - \sum_{i=1}^\ell N_i .
\end{equation*}
We assume $\ell \geq 2$ because any complete contraction of $\cnabla^N\cRm$ is proportional to $\cDelta^{N/2}\cR$ modulo ambient scalar Riemannian invariants, and $\cDelta^{N/2}\cR=0$ when it is independent of the ambiguity of $\cg$.
If $\cI$ is independent of the ambiguity of $\cg$, then $\mI := \iota^\ast \cI_{\cg} \in \mE[w]$ is independent of the choice of ambient space.
A \emph{scalar Weyl invariant} is a scalar invariant $\mI \in \mE[w]$ constructed in this way.
Fefferman and Graham gave a condition on the weight $w$ which implies this independence:

\begin{lemma}[Cf.~\cite{FeffermanGraham2012}]
	\label{fefferman-graham-lemma}
	Let $(\cmG^{n+2},\cg)$ be a straight and normal ambient space and let $\cI \in \cmE[w]$ be an ambient scalar Riemannian invariant.
	If $w \geq -n - 2$, then $\iota^\ast\cI_{\cg}$ is independent of the ambiguity of $\cg$.
\end{lemma}

Bailey, Eastwood, and Graham ~\cite[Theorem~A]{BaileyEastwoodGraham1994} showed that {\itshape every conformally invariant scalar of weight $w>-n$ is a Weyl invariant}.

\subsection{Conformally invariant polydifferential operators}
\label{subsec:ovsienko-redou}

Fix $k,n \in \bN$.
An ambient \emph{polydifferential operator} $\cD$ of weight $-2k$ is an assignment to each ambient space $(\cmG^{n+2},\cg)$ of a linear map
\begin{equation*}
	\cD^{\cg} \colon \cmE[w_1] \otimes \dotsm \otimes \cmE[w_j] \to \cmE[w_1 + \dotsm + w_j - 2k]
\end{equation*}
such that $\cD^{\cg}( \cu_1 \otimes \dotsm \otimes \cu_j)$ is a linear combination of complete contractions of
\begin{equation}
	\label{eqn:polydifferential-precontraction}
	\cnabla^{N_1}\cu_1 \otimes \dotsm \otimes \cnabla^{N_j}\cu_j \otimes \cnabla^{N_{j+1}}\cRm \otimes \dotsm \otimes \cnabla^{N_\ell}\cRm
\end{equation}
with $\ell=j$ or $\ell \geq j+2$.
Necessarily the powers $N_1,\dotsc,N_\ell$ satisfy
\begin{equation*}
	\sum_{i=1}^\ell N_i + 2\ell - 2j = 2k.
\end{equation*}
The \emph{total order} of such a contraction is $\sum_{i=1}^j N_i$.
We say that $\cD$ is \emph{tangential} if $\iota^\ast(\cD^{\cg}(\cu_1 \otimes \dotsm \otimes \cu_j))$ depends only on $\iota^\ast\cu_1,\dotsc,\iota^\ast\cu_j$ and $\cg$ modulo its ambiguity.
On each conformal manifold $(M^n,\kc)$, such an operator determines a \emph{conformally invariant polydifferential operator} 
\begin{equation*}
	D \colon \mE[w_1] \otimes \dotsm \otimes \mE[w_j] \to \mE[w_1 + \dotsm + w_j - 2k].
\end{equation*}

We now recall a condition on the total order of an ambient polydifferential operator that implies that it is independent of the ambiguity of $\cg$ (see \cite[Proposition~9.1]{FeffermanGraham2012}).

\begin{proposition}[Cf.~\cite{CaseYan2024}]
	\label{dependence}
	Let $(\cmG^{n+2},\cg)$ be a straight and normal ambient space and let $\cD$ be an ambient polydifferential operator of weight $-2k$.
	Suppose that
	\begin{enumerate}
		\item[(i)] $n$ is odd,
		\item[(ii.1)] $k \leq \frac{n}{2}$, or
		\item[(ii.2)] $k \leq \frac{n}{2} + 1$ and $\cD$ can be expressed as a linear combination of complete contractions of tensors of the form~\eqref{eqn:polydifferential-precontraction} with $\ell \geq j + 2$.
	\end{enumerate}
	Then $\cD$ is independent of the ambiguity of $\cg$.
\end{proposition}

Now let $D \colon \mE\bigl[-\frac{n-2k}{j+1}\bigr]^{\otimes j} \to \mE\bigl[-\frac{jn+2k}{j+1}\bigr]$ be a conformally invariant polydifferential operator.
Then for every compact conformal manifold $(M^n,\kc)$, the Dirichlet form $\kD \colon \mE\bigl[-\frac{n-2k}{j+1}\bigr]^{\otimes(j+1)} \to \bR$ determined by
\begin{equation*}
	\mathfrak{D}(u_0 \otimes \dotsm \otimes u_j) := \int_M u_0D(u_1 \otimes \dotsm \otimes u_j) \dvol ,
\end{equation*}
is conformally invariant.
We say that $D$ is \emph{formally self-adjoint} if $\kD$ is symmetric.
This implies that $D$ is itself symmetric.

We conclude this subsection by constructing the curved Ovsienko--Redou operators $D_{2k}$ and their linear analogues $D_{2k;\mI}$; see~\cite[Section~2]{CaseYan2024} or \cite[Lemma~6.1]{CaseLinYuan2022or} for more details.

\begin{lemma}
	\label{linear}
	Let $(M^n,\kc)$ be a conformal manifold and let $\cI \in \cmE[-2\ell]$ be an ambient scalar Riemannian invariant.
	Let $k \in \bN$ and if $n$ is even, we assume additionally that $k+\ell\le \frac{n}{2}+1-\delta_{0,\ell}$ with $\delta$ the Kronecker delta. Then
	\begin{equation*}
		\cD_{2k;\cI}(\cu) := \sum_{r+s=k} \frac{k!}{r!s!}\frac{(\ell+s-1)!(\ell+r-1)!}{(\ell-1)!^2} \cDelta^r\left( \cI \cDelta^s\cu \right)
	\end{equation*}
	defines a tangential differential operator $\cD_{2k;\cI} \colon \cmE\bigl[-\frac{n-2k-2\ell}{2}\bigr] \to \cmE\bigl[-\frac{n+2k+2\ell}{2}\bigr]$.
	In particular, the differential operator $D_{2k;\mI} \colon \mE\bigl[ -\frac{n-2k-2\ell}{2} \bigr] \to \mE\bigl[ -\frac{n+2k+2\ell}{2} \bigr]$ defined by
	\begin{align*}
		D_{2k;\mI}(\iota^\ast\cu) := \big(\cD_{2k;\cI}\cu\big) \circ \iota ,
	\end{align*}
	is conformally invariant.
\end{lemma}

The curved Ovsienko--Redou operators arise by looking for tangential linear combinations of the operators $\cD_{2k-2s;\cDelta^{s}\cf}$.

\begin{lemma}\label{le:cD-2k}
	\label{ovsienko-redou}
	Let $(M^n,\kc)$ be a conformal manifold.
	Let $k \in \bN$ and
	if $n$ is even, we assume additionally that $k \leq \frac{n}{2}$.
	Then
	\begin{align*}
		\cD_{2k}(\cu \otimes \cv) := \sum_{r+s+t=k} a_{r,s,t} \cDelta^r\left( (\cDelta^s\cu)(\cDelta^t\cv) \right)
	\end{align*}
	with
	\begin{align*}
		a_{r,s,t} := \frac{k!}{r!s!t!}\frac{\Gamma\bigl(\frac{n+4k}{6}-r\bigr)\Gamma\bigl(\frac{n+4k}{6}-s\bigr)\Gamma\bigl(\frac{n+4k}{6}-t\bigr)}{\Gamma\bigl(\frac{n-2k}{6}\bigr)\Gamma\bigl(\frac{n+4k}{6}\bigr)^2}
	\end{align*}
	defines a tangential bidifferential operator $\cD_{2k} \colon \cmE\bigl[-\frac{n-2k}{3}\bigr]^{\otimes 2} \to \cmE\bigl[-\frac{2n+2k}{3}\bigr]$.
	In particular, the bidifferential operator $D_{2k} \colon \mE\bigl[ -\frac{n-2k}{3} \bigr]^{\otimes 2} \to \mE\bigl[ -\frac{2n+2k}{3} \bigr]$ defined by
	\begin{align*}
		D_{2k}( \iota^\ast\cu \otimes \iota^\ast\cv ) := \cD_{2k}( \cu \otimes \cv ) \circ \iota
	\end{align*}
	is conformally invariant.
\end{lemma}

According to the above information, it can be seen that the operators
\begin{align}\label{generalized-GJMS}
	D_{M,L}(u):= \mR_{L+1-2M} \cdots \mR_{L-3} \mR_{L-1} (u)\Big\vert_{\rho=0}
\end{align}
that will play a fundamental role in our analysis may also be realized by 
\begin{align}\label{formal-GJMS}
	D_{M,L}(u)=\cDelta^M\left(\cu \right)\Big\vert_{\tau=1, \rho=0},
\end{align}
where $\cu(\tau,x,\rho)=\tau^{L-\frac{n}{2}}u(x)$ is a conformal density of weight $L-\frac{n}{2}$ on $\cmG$. However, operators defined in \eqref{formal-GJMS} may depend on the extension of $u$ on $M^n\times (-\epsilon, \epsilon)$. So it is more convenient to work with the definition \eqref{generalized-GJMS} which avoids the discussion on the tangential property.
\section{Diffindo}\label{diffindo}

In Subsection~\ref{sec:diffindo-intro}, we have pointed out that the key in our argument is to split the analyses of the inner and outer layers of operator compositions in $D_{2k}$ and $D_{2k;\mI}$. Particularly, what are produced from the inner layer are essentially $\rho^N$ or $\rho^N f$ with $N\in \bN$ and $f=f(\rho)$. Since the outer layer is simply of the form $\mR_{L+1-2M} \cdots \mR_{L-3} \mR_{L-1}$, we shall look at $D_{M,L}(\rho^N)$ and $D_{M,L}(\rho^N f)$ to cast the \emph{Diffindo}. In view of \eqref{eq:SAB-def}, we first need to evaluate $\mathcal{S}_{\mathbf{A},\mathbf{B},M,L}(\rho^N)$ and $\mathcal{S}_{\mathbf{A},\mathbf{B},M,L}(\rho^N f)$ for arbitrary sequences $\mathbf{A}=(A_1,\ldots,A_r)$ and $\mathbf{B}=(B_0,B_1,\ldots,B_r)$ of nonnegative integers such that
\begin{align*}
	r + \sum_{j=0}^r B_j = M.
\end{align*}

\subsection{Evaluation of $D_{M,L}(\rho^N)$}

We begin with $\mathcal{S}_{\mathbf{A},\mathbf{B},M,L}(\rho^N)$. To simplify our notation, we write
\begin{align}
	S_{\mathbf{A},\mathbf{B}}(N) = S_{\mathbf{A},\mathbf{B},M,L}(N) := \mathcal{S}_{\mathbf{A},\mathbf{B},M,L}(\rho^N).
\end{align}
The following result gives an explicit expression of $S_{\mathbf{A},\mathbf{B},M,L}(N)$.

\begin{lemma}\label{le:S(A,B)}
	For any nonnegative integer $N$,
	\begin{align}\label{eq:S(A,B)}
		S_{\mathbf{A},\mathbf{B},M,L}(N)& = \rho^{N+\sum\limits_{j=1}^r A_j - \sum\limits_{j=0}^r B_j} \prod_{i=0}^{r} \big(B_i!\big)^2 \\
		&\quad\times \scalebox{0.8}{$\left(\begin{array}{c}
				N+\sum\limits_{j=1}^i A_j - \sum\limits_{j=0}^{i-1} B_j\\[12pt]
				B_i
			\end{array}\right)$}
		\scalebox{0.8}{$\left(\begin{array}{c}
				L-N-2i-\sum\limits_{j=1}^i A_j - \sum\limits_{j=0}^{i-1} B_j\\[12pt]
				B_i
			\end{array}\right)$}.\notag
	\end{align}
	In addition, $S_{\mathbf{A},\mathbf{B},M,L}(N)$ vanishes if there exists a certain index $i$ such that
	\begin{align}\label{eq:S(A,B)=0-condition}
		N+\sum\limits_{j=1}^i A_j < \sum\limits_{j=0}^i B_j.
	\end{align}
\end{lemma}

\begin{proof}
	For arbitrary $\ell$ and $n$, we note that
	\begin{align*}
		\mD_\ell(\rho^n) = \rho^{n-1} \cdot n(\ell+1-n).
	\end{align*}
	Hence,
	\begin{align*}
		\mD_{\ell+1-2b}\cdots \mD_{\ell-3}\mD_{\ell-1} (\rho^n) &= \rho^{n-b}\cdot \prod_{k=0}^{b-1} (n-k)(\ell-n-k)\\
		&= \rho^{n-b}\cdot \big(b!\big)^2 \binom{n}{b}\binom{\ell-n}{b}.
	\end{align*}
	Repeatedly applying the above argument yields \eqref{eq:S(A,B)}.
	
	For the second part, it is trivial when $M=0$ since in this case no operator is acted on $\rho^N$. For $M\ge 1$, let $i$ be the smallest index such that \eqref{eq:S(A,B)=0-condition} holds. It follows that
	\begin{align*}
		B_i > N+\sum_{j=1}^i A_j - \sum_{j=0}^{i-1} B_j.
	\end{align*}
	Furthermore, if $i=0$,
	\begin{align*}
		N+\sum_{j=1}^i A_j - \sum_{j=0}^{i-1} B_j = N \ge 0.
	\end{align*}
	Otherwise, the fact that $i$ is the smallest index ensuring \eqref{eq:S(A,B)=0-condition} implies
	\begin{align*}
		N+\sum_{j=1}^i A_j - \sum_{j=0}^{i-1} B_j = A_i + \left(N + \sum_{j=1}^{i-1} A_j - \sum_{j=0}^{i-1} B_j\right) \ge A_i \ge 0.
	\end{align*}
	Hence, we have the vanishing of the binomial coefficient
	\begin{align*}
		\scalebox{0.8}{$\left(\begin{array}{c}
				N+\sum\limits_{j=1}^i A_j - \sum\limits_{j=0}^{i-1} B_j\\[12pt]
				B_i
			\end{array}\right)$},
	\end{align*}
	thereby implying that $S_{\mathbf{A},\mathbf{B}}(N)$ also vanishes.
\end{proof}

The above evaluation immediately leads us to an explicit expression of $D_{M,L}\circ \mP_N (u)$ for $u$ independent of $\rho$, which further gives $D_{M,L}(\rho^N)$ by taking $u\equiv 1$. More importantly, this result, as shown in the next corollary, serves as a generalization of Juhl's formula in Theorem \ref{Juhl-formula}.

\begin{corollary}\label{coro:D-u}
	The operator $D_{M,L}\circ \mP_N$ satisfies
	\begin{align}\label{generalized-Juhl-formula}
		&D_{M,L}\circ \mP_N(u)\\
		&\quad= \sum_{\mathbf{A}} \mM_{2(A_r+1)}\cdots \mM_{2(A_1+1)} (u)\cdot  (-1)^{M-N-r}\, 2^N \prod_{i=1}^r \frac{1}{(A_{i}!)^2}\notag\\
		&\quad\quad\times \big(M!\big)\prod_{n=0}^{M-1}(L-M-n)\prod_{i=1}^r \scalebox{0.75}{$\dfrac{1}{\sum\limits_{j=1}^i (A_{r+1-j}+1)}$}\prod_{i=1}^r \scalebox{0.75}{$\dfrac{1}{L-2M+\sum\limits_{j=1}^i (A_{r+1-j}+1)}$},\notag
	\end{align}
	where the summation runs over all sequences $\mathbf{A}=(A_{1},\ldots,A_{r})$ of nonnegative integers such that
	\begin{align}\label{eq:generalized-Juhl-formula-cond}
		r + \sum_{j=1}^r A_j = M-N, \quad\quad 0\leq N\le M.
	\end{align}
\end{corollary}

\begin{proof}
	We start by writing $\cD_{M,L}(\rho^N u)$ as
	\begin{align*}
		\cD_{M,L}(\rho^N u) & = \sum_{\mathbf{A}}\sum_{\mathbf{B}} \mM_{2(A_r+1)}\cdots \mM_{2(A_1+1)} \big(S_{\mathbf{A},\mathbf{B}}(N)u\big)\\
		&\quad\times (-1)^{\sum\limits_{j=1}^r A_j} 2^{-\sum\limits_{j=1}^r A_j + \sum\limits_{j=0}^r B_j} \prod_{i=1}^r \frac{1}{(A_{i}!)^2}.
	\end{align*}
	In view of \eqref{eq:B-cond}, $\mathbf{B}=(B_0,B_1,\ldots,B_r)$ is such that
	\begin{align}\label{eq:B-cond-1}
		r + \sum_{j=0}^r B_j = M.
	\end{align}
	In addition, since $D_{M,L}=\cD_{M,L}\big\vert_{\rho=0}$, we further require that the power of $\rho$ in $S_{\mathbf{A},\mathbf{B}}(N)$ reduces to zero. By \eqref{eq:S(A,B)}, we have
	\begin{align*}
		N+\sum_{j=1}^r A_j - \sum_{j=0}^r B_j = 0,
	\end{align*}
	so that $\mathbf{A}=(A_1,\ldots,A_r)$ satisfies
	\begin{align}\label{eq:A-cond-1}
		r + \sum_{j=1}^r A_j = M-N.
	\end{align}
	Running over nonnegative sequences $\mathbf{A}$ and $\mathbf{B}$ restricted as above and invoking \eqref{eq:S(A,B)}, we have
	\begin{align*}
		&D_{M,L}(\rho^N u)\\
		&\quad= \sum_{\mathbf{A}} \mM_{2(A_r+1)}\cdots \mM_{2(A_1+1)} (u)\cdot  (-1)^{M-N-r}\, 2^N \prod_{i=1}^r \frac{1}{(A_{i}!)^2}\\
		&\quad\quad\times \sum_{\mathbf{B}} \prod_{i=0}^{r} \big(B_i!\big)^2 \scalebox{0.8}{$\left(\begin{array}{c}
				N+\sum\limits_{j=1}^i A_j - \sum\limits_{j=0}^{i-1} B_j\\[12pt]
				B_i
			\end{array}\right)$}
		\scalebox{0.8}{$\left(\begin{array}{c}
				L-N-2i-\sum\limits_{j=1}^i A_j - \sum\limits_{j=0}^{i-1} B_j\\[12pt]
				B_i
			\end{array}\right)$}.
	\end{align*}
	For the inner summation on $\mathbf{B}$, we make use of \eqref{eq:aux-sum-1} with the substitutions:
	\begin{align*}
		\mathbf{A} &\mapsto (N,A_1,\ldots,A_r),\\
		r &\mapsto r+1,\\
		N &\mapsto M+1,\\
		X &\mapsto L+2.
	\end{align*}
	Hence,
	\begin{align*}
		&\sum_{\mathbf{B}} \prod_{i=0}^{r} \big(B_i!\big)^2 \scalebox{0.8}{$\left(\begin{array}{c}
				N+\sum\limits_{j=1}^i A_j - \sum\limits_{j=0}^{i-1} B_j\\[12pt]
				B_i
			\end{array}\right)$}
		\scalebox{0.8}{$\left(\begin{array}{c}
				L-N-2i-\sum\limits_{j=1}^i A_j - \sum\limits_{j=0}^{i-1} B_j\\[12pt]
				B_i
			\end{array}\right)$}\\
		&\quad = \big(M!\big)\prod_{n=0}^{M-1}(L-M-n)\prod_{i=1}^r \scalebox{0.75}{$\dfrac{1}{\sum\limits_{j=1}^i (A_{r+1-j}+1)}$}\prod_{i=1}^\ell \scalebox{0.75}{$\dfrac{1}{L-2M+\sum\limits_{j=1}^i (A_{r+1-j}+1)}$},
	\end{align*}
	where we have also used \eqref{eq:B-cond-1} and \eqref{eq:A-cond-1}. This finishes the proof of \eqref{generalized-Juhl-formula}.
\end{proof}

\subsection{Evaluation of $D_{M,L}(\rho^N f)$}

Now we consider $\mathcal{S}_{\mathbf{A},\mathbf{B},M,L}(\rho^N f)$ for general $f=f(\rho)$.

\begin{lemma}\label{le:S(A,B)-f}
	For any nonnegative integer $N$ and any smooth function $f=f(\rho)$,
	\begin{align}\label{eq:S(A,B)-f}
		&\mathcal{S}_{\mathbf{A},\mathbf{B},M,L}(\rho^N f)\\
		&\quad = \sum_{R\ge 0}\rho^{R+N+\sum\limits_{j=1}^r A_j - \sum\limits_{j=0}^r B_j} f^{(R)}\cdot \frac{1}{R!} \sum_{l=0}^R (-1)^{R-l} \binom{R}{l}\notag\\
		&\quad\quad\times \prod_{i=0}^{r} \big(B_i!\big)^2 \scalebox{0.8}{$\left(\begin{array}{c}
				N+r+\sum\limits_{j=1}^i A_j - \sum\limits_{j=0}^{i-1} B_j\\[12pt]
				B_i
			\end{array}\right)$}
		\scalebox{0.8}{$\left(\begin{array}{c}
				L-N-r-2i-\sum\limits_{j=1}^i A_j - \sum\limits_{j=0}^{i-1} B_j\\[12pt]
				B_i
			\end{array}\right)$},\notag
	\end{align}
	where $f^{(R)}$ is the $R$-th order derivative with respect to $\rho$. In addition, the above summand vanishes for all $R>2\sum\limits_{j=0}^r B_j$.
\end{lemma}

\begin{proof}
	For arbitrary $\ell$ and $n$,
	\begin{align*}
		\mD_\ell(\rho^n f) = \rho^{n-1}f \cdot n(\ell+1-n) + \rho^n f'\cdot (\ell-2n) + \rho^{n+1}f''\cdot (-1).
	\end{align*}
	We observe that after applying the $\mD$-operator, the derivative order of $f$ changes by $i\in\{0,1,2\}$, and accordingly the power of $\rho$ changes by $i-1$. In $\mathcal{S}_{\mathbf{A},\mathbf{B}}$, we are to have $\sum_{j=0}^r B_j$ applications of $\mD$, and hence the derivative order of $f$ ranges over the interval $\big[0,2\sum_{j=0}^r B_j\big]$. This, in particular, confirms the second part of our result.
	
	Now we write
	\begin{align}\label{eq:S(A,B)-f-expansion}
		\mathcal{S}_{\mathbf{A},\mathbf{B}}(\rho^N f) = \sum_{R\ge 0} c_R\cdot \rho^{R+N+\sum\limits_{j=1}^r A_j - \sum\limits_{j=0}^r B_j} f^{(R)}.
	\end{align}
	For the evaluation of the coefficients $c_R$, our trick is to choose $f(\rho) = \frac{1}{R!}(\rho-1)^R$. The takeaway is that the expression
	\begin{align*}
		\partial^l_\rho\, \tfrac{1}{R!}(\rho-1)^R \Big\vert_{\rho=1}
	\end{align*}
	equals $1$ when $l=R$, and vanishes for all other $l$. Hence,
	\begin{align*}
		c_R &= \sum_{l\ge 0} c_l\cdot \partial^l_\rho\, \tfrac{1}{R!}(\rho-1)^R \Big\vert_{\rho=1}\\
		&= \sum_{l\ge 0} c_l\cdot \rho^{l+N+\sum\limits_{j=1}^r A_j - \sum\limits_{j=0}^r B_j} \partial^l_\rho\, \tfrac{1}{R!}(\rho-1)^R \Bigg\vert_{\rho=1}\\
		&= \mathcal{S}_{\mathbf{A},\mathbf{B}}\big(\rho^N \tfrac{1}{R!}(\rho-1)^R\big) \Big\vert_{\rho=1}\\
		&= \frac{1}{R!}\sum_{l=0}^R (-1)^{R-l} \binom{R}{l} \mathcal{S}_{\mathbf{A},\mathbf{B}}(\rho^{N+l}) \Bigg\vert_{\rho=1}.
	\end{align*}
	Invoking \eqref{eq:S(A,B)} gives the desired relation.
\end{proof}

As a consequence, we also arrive at an explicit expression of $D_{M,L}\circ \mP_N (fu)$ for $u$ independent of $\rho$, which reduces to $D_{M,L}(\rho^N f)$ by taking $u\equiv 1$.

\begin{corollary}\label{D-f-u}
	For any nonnegative integer $N$ and any smooth function $f=f(\rho)$,
	\begin{align}\label{eq:D-f-u}
		&D_{M,L}\circ \mP_N( f u)\\
		&= \sum_{R}\sum_{\mathbf{A}} \mM_{2(A_r+1)}\cdots \mM_{2(A_1+1)} (f^{(R)}u)\cdot  (-1)^{M-N-R-r}\, 2^{N+R} \frac{1}{R!}\prod_{i=1}^r \frac{1}{(A_{i}!)^2}\notag\\
		&\quad\times \big(M!\big)\prod_{n=0}^{M-1}(L-M-n)\prod_{i=1}^r \scalebox{0.75}{$\dfrac{1}{\sum\limits_{j=1}^i (A_{r+1-j}+1)}$}\prod_{i=1}^r \scalebox{0.75}{$\dfrac{1}{L-2M+\sum\limits_{j=1}^i (A_{r+1-j}+1)}$},\notag
	\end{align}
	where the summation runs over all nonnegative integers $R$ and all sequences $\mathbf{A}=(A_{1},\ldots,A_{r})$ of nonnegative integers such that
	\begin{align}\label{eq:R-cond}
		R + r + \sum_{j=1}^r A_j = M-N.
	\end{align}
	Here by abuse of notation, we read $f^{(R)}$ as $f^{(R)}\big|_{\rho=0}$.
\end{corollary}

\begin{proof}
	Note that $\cD_{M,L}\circ \mP_N( f u)$ can be written as
	\begin{align*}
		&\cD_{M,L}\circ \mP_N(f u)\\
		&\quad = \sum_{\mathbf{A}}\sum_{\mathbf{B}} \mM_{2(A_r+1)}\cdots \mM_{2(A_1+1)} \big(\mathcal{S}_{\mathbf{A},\mathbf{B},M,L}(\rho^N f)u\big) \\
		&\quad\quad\times (-1)^{\sum\limits_{j=1}^r A_j} 2^{-\sum\limits_{j=1}^r A_j + \sum\limits_{j=0}^r B_j} \prod_{i=1}^r \frac{1}{(A_{i}!)^2}\\
		&\quad=\sum_{\mathbf{A}}\sum_{\mathbf{B}}\sum_{R} \mM_{2(A_r+1)}\cdots \mM_{2(A_1+1)}\left(\rho^{R+N+\sum\limits_{j=1}^r A_j - \sum\limits_{j=0}^r B_j} f^{(R)}(\rho) \,u\right)\\
		&\quad\quad\times (-1)^{\sum\limits_{j=1}^r A_j} 2^{-\sum\limits_{j=1}^r A_j + \sum\limits_{j=0}^r B_j} \prod_{i=1}^r \frac{1}{(A_{i}!)^2} \cdot c_R,
	\end{align*}
	where we have utilized \eqref{eq:S(A,B)-f-expansion}. Still, $\mathbf{B}=(B_0,B_1,\ldots,B_\ell)$ is such that
	\begin{align}\label{eq:B-cond-2}
		r + \sum_{j=0}^r B_j = M.
	\end{align}
	Meanwhile, we have shown in the proof of Lemma \ref{le:S(A,B)-f} that
	\begin{align*}
		c_R = \frac{1}{R!}\sum_{l=0}^R (-1)^{R-l} \binom{R}{l} S_{\mathbf{A},\mathbf{B}}(N+l) \bigg\vert_{\rho=1}.
	\end{align*}
	Thus, if $R$ is such that
	\begin{align*}
		R+N+\sum\limits_{j=1}^r A_j - \sum\limits_{j=0}^r B_j < 0,
	\end{align*}
	we must have the vanishing of $S_{\mathbf{A},\mathbf{B}}(N+l)$ for all $0\le l\le R$ by the second part of Lemma \ref{le:S(A,B)}, and thus the vanishing of $c_R$. Now for $D_{M,L}=\cD_{M,L}\big\vert_{\rho=0}$, it suffices to restrict
	\begin{align}\label{eq:R-cond-2}
		R+N+\sum\limits_{j=1}^r A_j - \sum\limits_{j=0}^r B_j=0
	\end{align}
	so that $\mathbf{A}=(A_1,\ldots,A_r)$ and $R$ satisfy
	\begin{align}\label{eq:A-R-cond-2}
		R + r + \sum_{j=1}^r A_j = M-N.
	\end{align}
	With the additional restriction for $R$ in \eqref{eq:R-cond-2}, we further find that when $l<R$,
	\begin{align*}
		l+N+\sum\limits_{j=1}^r A_j < \sum\limits_{j=0}^r B_j,
	\end{align*}
	which, in light of the second part of Lemma \ref{le:S(A,B)}, implies the vanishing of $S_{\mathbf{A},\mathbf{B}}(N+l)$ for these $l$. Hence, for $R$ restricted by \eqref{eq:R-cond-2}, we always have
	\begin{align*}
		c_R &= \frac{1}{R!} S_{\mathbf{A},\mathbf{B}}(N+R) \Big\vert_{\rho=1}\\
		&= \frac{1}{R!}\prod_{i=0}^{r} \big(B_i!\big)^2 \scalebox{0.8}{$\left(\begin{array}{c}
				N+R+\sum\limits_{j=1}^i A_j - \sum\limits_{j=0}^{i-1} B_j\\[12pt]
				B_i
			\end{array}\right)$}
		\scalebox{0.8}{$\left(\begin{array}{c}
				L-N-R-2i-\sum\limits_{j=1}^i A_j - \sum\limits_{j=0}^{i-1} B_j\\[12pt]
				B_i
			\end{array}\right)$},
	\end{align*}
	by invoking \eqref{eq:S(A,B)}. Now running over nonnegative integers $R$ and nonnegative sequences $\mathbf{A}$ and $\mathbf{B}$ as restricted by \eqref{eq:B-cond-2} and \eqref{eq:A-R-cond-2}, we have
	\begin{align*}
		&D_{M,L}\circ \mP_N( f u)\\
		&= \sum_R \sum_{\mathbf{A}} \mM_{2(A_r+1)}\cdots \mM_{2(A_1+1)} (f^{(R)}u)\cdot  (-1)^{M-N-R-r}\, 2^{N+R} \frac{1}{R!} \prod_{i=1}^r \frac{1}{(A_{i}!)^2}\\
		&\quad\times \sum_{\mathbf{B}} \prod_{i=0}^{r} \big(B_i!\big)^2 \scalebox{0.8}{$\left(\begin{array}{c}
				N+R+\sum\limits_{j=1}^i A_j - \sum\limits_{j=0}^{i-1} B_j\\[12pt]
				B_i
			\end{array}\right)$}
		\scalebox{0.8}{$\left(\begin{array}{c}
				L-N-R-2i-\sum\limits_{j=1}^i A_j - \sum\limits_{j=0}^{i-1} B_j\\[12pt]
				B_i
			\end{array}\right)$}.
	\end{align*}
	Note that the inner summation on $\mathbf{B}$ is exactly the one in the proof of Corollary \ref{coro:D-u} with $N$ replaced by $N+R$. The claimed result therefore follows.
\end{proof}
\section{Juhl's formula revisited}\label{sec:FG}

By taking $N= 0$ and $L=M=k$ in $D_{M,L}\circ \mP_N(u)$, we have
\begin{align}
	D_{k,k}\circ \mP_0(u) = \mR_{1-k}\mR_{3-k}\cdots \mR_{k-3}\mR_{k-1} (u)\Big\vert_{\rho=0} = P_{2k}(u),
\end{align}
which is exactly the GJMS operator of order $2k$. Applying this specialization to Corollary \ref{coro:D-u}, we immediately have Juhl's formula in Theorem \ref{Juhl-formula}.

Now a natural question is that {\itshape are there other formally self-adjoint GJMS operators?} We shall give an answer in the next theorem.

\begin{theorem}\label{th:GJMS-self-adj}
	$D_{M,L}\circ \mP_N$ is a formally self-adjoint operator if and only if $L=M+N$, which is the GJMS operator of order $2(M-N)$ up to a constant.
\end{theorem}

\begin{proof}
	If $D_{M,L}$ is formally self-adjoint, from the self-adjointness of each $\mM_{2(N+1)}$, it follows that the coefficient $\mC_{\mathbf{A}}$ of $\mM_{2(A_r+1)}\cdots \mM_{2(A_1+1)}$ should satisfy $\mC_{\mathbf{A}}=\mC_{\mathbf{A}^{-1}}$ with $\mathbf{A}^{-1}=(A_r,\ldots,A_1)$, which implies that 
	\begin{align*}
		L=M+N.
	\end{align*}
	In this case, for $1\le i\le r-1$,
	\begin{align*}
		L-2M+\sum\limits_{j=1}^i (A_{r+1-j}+1)=-\sum\limits_{j=i+1}^r (A_{r+1-j}+1),
	\end{align*}
	and
	\begin{align*}
		L-2M+\sum\limits_{j=1}^r (A_{r+1-j}+1)=L-M-N
	\end{align*}
	will be cancelled by $\prod_{n=0}^{M-1}(L-M-n)$. The desired conclusion follows immediately.
\end{proof}
\section{Formal self-adjointness of $D_{2k}$}\label{sec:selfadj-2}

Recall from Section~\ref{sec:bg} that the curved Ovsienko--Redou operators $D_{2k}$ are determined ambiently by
\begin{align*}
    D_{2k}(u\otimes v):=\cD_{2k}(\cu \otimes \cv) \Big\vert_{\tau=1, \rho=0},
\end{align*}
where, in light of Lemma \ref{le:cD-2k},
\begin{align*}
	\cD_{2k}(\cu \otimes \cv) = \sum_{r+s+t=k} a_{r,s,t} \cDelta^r\left( (\cDelta^s\cu)(\cDelta^t\cv) \right)
\end{align*}
with
\begin{align*}
	\cu=\tau^{\gamma_k}u,\qquad  \cv=\tau^{\gamma_k}v.
\end{align*}
Here
\begin{align*}
	\gamma_k=-\frac{n}{3}+\frac{2k}{3}.
\end{align*}
For convenience, we further set
\begin{equation*}
	L_k:=\gamma_k+\frac{n}{2} = \frac{n}{6}+\frac{2k}{3},
\end{equation*}
which simplifies $a_{r,s,t}$ as
\begin{align*}
	a_{r,s,t} = \binom{k}{r,s,t}\frac{\Gamma(L_k-r)\Gamma(L_k-s)\Gamma(L_k-t)}{\Gamma(L_k-k)\Gamma(L_k)^2}.
\end{align*}
In view of \eqref{Laplacian-expansion}, we rewrite the $(\cDelta\vert_{\tau=1})$-operators in terms of the $\mR$-operators and arrive at
\begin{align}
	&D_{2k}(u \otimes v)\\
	&\qquad= \sum_{r+s+t=k} a_{r,s,t} \mR_{-L_k+1}\cdots \mR_{-L_k+2r-3}\mR_{-L_k+2r-1}\notag\\
	&\qquad\quad \big( \mR_{L_k+1-2s}\cdots \mR_{L_k-3}\mR_{L_k-1} (u)\mR_{L_k+1-2t}\cdots \mR_{L_k-3}\mR_{L_k-1} (v)\big)\Big\vert_{\rho=0}.\notag
\end{align}
This was already presented in \eqref{eq:D-2k}. It is also clear that $D_{2k}$ can be specialized from
\begin{align*}
	&D_{U,V,L,K^\ssfu,K^\ssfv}(u\otimes v)\\
	&\qquad:= \sum_{\substack{M^{\ssfu},M^{\ssfv},M'\ge 0\\ M^{\ssfu}+M^{\ssfv}+M'=U}} \frac{\Gamma(U+K^\ssfu+1)\Gamma(U+K^\ssfv+1)}{\Gamma(M^\ssfu+K^\ssfu+1)\Gamma(M^\ssfv+K^\ssfv+1)\Gamma(M'+1)}\notag\\
	&\qquad\ \quad\times \frac{\Gamma(L-M^\ssfu)\Gamma(L-M^\ssfv)\Gamma(L+V-M')}{\Gamma(L-U)\Gamma(L)^2}\notag\\
	&\qquad\ \quad\times D_{[M',-L-V+2M'],[M^{\ssfu},L-K^{\ssfu}],[M^{\ssfv},L-K^{\ssfv}]}(u\otimes v)\notag
\end{align*}
by observing that
\begin{align}
	D_{2k}(u\otimes v) = \frac{1}{k!}D_{k,0,L_k,0,0}(u\otimes v).
\end{align}

To expand $D_{U,V,L,K^\ssfu,K^\ssfv}$, we begin with the operators
\begin{align*}
	D_{[M',L'],[M^{\ssfu},L^{\ssfu}],[M^{\ssfv},L^{\ssfv}]}(u\otimes v) := \cD_{M',L'}\big(\cD_{M^{\ssfu},L^{\ssfu}}(u)\cD_{M^{\ssfv},L^{\ssfv}}(v)\big)\Big\vert_{\rho=0}.
\end{align*}
Recall that $u$ and $v$ are smooth functions independent of $\rho$ and let $N^{\ssfu}$ and $N^{\ssfv}$ be nonnegative integers.

For the inner layer, we have two multiplicands, and we evaluate them separately. By \eqref{eq:S-prefactor} and \eqref{eq:S(A,B)},
\begin{align*}
	&\mR_{L^{\ssfu}+1-2M^{\ssfu}} \cdots \mR_{L^{\ssfu}-3} \mR_{L^{\ssfu}-1} (\rho^{N^{\ssfu}} u)\\
	&\quad = \sum_{\mathbf{A}^{\ssfu}}\mM_{2(A_{r^{\ssfu}}^{\ssfu}+1)}\cdots \mM_{2(A_1^{\ssfu}+1)}(u) \cdot (-1)^{\sum\limits_{j=1}^{r^{\ssfu}} A_j^{\ssfu}} 2^{-\sum\limits_{j=1}^{r^{\ssfu}} A_j^{\ssfu} + \sum\limits_{j=0}^{r^{\ssfu}} B_j^{\ssfu}} \prod_{i=1}^{r^{\ssfu}} \frac{1}{(A_{i}^{\ssfu}!)^2}\\
	&\quad\quad\times\sum_{\mathbf{B}^{\ssfu}}  \prod_{i=0}^{r^{\ssfu}} \big(B_i^{\ssfu}!\big)^2 \scalebox{0.8}{$\left(\begin{array}{c}
			N^{\ssfu}+\sum\limits_{j=1}^i A_j^{\ssfu} - \sum\limits_{j=0}^{i-1} B_j^{\ssfu}\\[12pt]
			B_i^{\ssfu}
		\end{array}\right)$}
	\scalebox{0.8}{$\left(\begin{array}{c}
			L^{\ssfu}-N^{\ssfu}-2i-\sum\limits_{j=1}^i A_j^{\ssfu} - \sum\limits_{j=0}^{i-1} B_j^{\ssfu}\\[12pt]
			B_i^{\ssfu}
		\end{array}\right)$}\\
	&\quad\quad\times \rho^{N^{\ssfu}+\sum\limits_{j=1}^{r^{\ssfu}} A_j^{\ssfu} - \sum\limits_{j=0}^{r^{\ssfu}} B_j^{\ssfu}},
\end{align*}
in which the nonnegative sequence $\mathbf{B}^{\ssfu}=(B_0^{\ssfu},B_1^{\ssfu},\ldots,B_{r^{\ssfu}}^{\ssfu})$ is such that
\begin{align}\label{eq:B-cond-D2k-u}
	r^{\ssfu} + \sum_{j=0}^{r^{\ssfu}} B_j^{\ssfu} = M^{\ssfu},
\end{align}
in light of \eqref{eq:B-cond}. Similarly,
\begin{align*}
	&\mR_{L^{\ssfv}+1-2M^{\ssfv}} \cdots \mR_{L^{\ssfv}-3} \mR_{L^{\ssfv}-1} (\rho^{N^{\ssfv}} v)\\
	&\quad = \sum_{\mathbf{A}^{\ssfv}}\mM_{2(A_{r^{\ssfv}}^{\ssfv}+1)}\cdots \mM_{2(A_1^{\ssfv}+1)}(v) \cdot (-1)^{\sum\limits_{j=1}^{r^{\ssfv}} A_j^{\ssfv}} 2^{-\sum\limits_{j=1}^{r^{\ssfv}} A_j^{\ssfv} + \sum\limits_{j=0}^{r^{\ssfv}} B_j^{\ssfv}} \prod_{i=1}^{r^{\ssfv}} \frac{1}{(A_{i}^{\ssfv}!)^2}\\
	&\quad\quad\times\sum_{\mathbf{B}^{\ssfv}}  \prod_{i=0}^{r^{\ssfv}} \big(B_i^{\ssfv}!\big)^2 \scalebox{0.8}{$\left(\begin{array}{c}
			N^{\ssfv}+\sum\limits_{j=1}^i A_j^{\ssfv} - \sum\limits_{j=0}^{i-1} B_j^{\ssfv}\\[12pt]
			B_i^{\ssfv}
		\end{array}\right)$}
	\scalebox{0.8}{$\left(\begin{array}{c}
			L^{\ssfv}-N^{\ssfv}-2i-\sum\limits_{j=1}^i A_j^{\ssfv} - \sum\limits_{j=0}^{i-1} B_j^{\ssfv}\\[12pt]
			B_i^{\ssfv}
		\end{array}\right)$}\\
	&\quad\quad\times \rho^{N^{\ssfv}+\sum\limits_{j=1}^{r^{\ssfv}} A_j^{\ssfv} - \sum\limits_{j=0}^{r^{\ssfv}} B_j^{\ssfv}},
\end{align*}
where the nonnegative sequence $\mathbf{B}^{\ssfv}=(B_0^{\ssfv},B_1^{\ssfv},\ldots,B_{r^{\ssfv}}^{\ssfv})$ is such that
\begin{align}\label{eq:B-cond-D2k-v}
	r^{\ssfv} + \sum_{j=0}^{r^{\ssfv}} B_j^{\ssfv} = M^{\ssfv}.
\end{align}

For the outer layer, we are essentially looking at
\begin{align*}
	&\mR_{L'+1-2M'} \cdots \mR_{L'-3} \mR_{L'-1}\left(\rho^{N^{\ssfu}+N^{\ssfv}+\sum\limits_{j=1}^{r^{\ssfu}} A_j^{\ssfu} +\sum\limits_{j=1}^{r^{\ssfv}} A_j^{\ssfv} - \sum\limits_{j=0}^{r^{\ssfu}} B_j^{\ssfu} - \sum\limits_{j=0}^{r^{\ssfv}} B_j^{\ssfv}}\right)\Bigg|_{\rho=0}\\
	&\quad= \sum_{\mathbf{A}'} \mM_{2(A'_{r'}+1)}\cdots \mM_{2(A'_1+1)} (1)\\
	&\quad\quad\times  (-1)^{M'-N^{\ssfu}-N^{\ssfv}-\sum\limits_{j=1}^{r^{\ssfu}} A_j^{\ssfu} -\sum\limits_{j=1}^{r^{\ssfv}} A_j^{\ssfv} + \sum\limits_{j=0}^{r^{\ssfu}} B_j^{\ssfu} + \sum\limits_{j=0}^{r^{\ssfv}} B_j^{\ssfv}-r'}\\
	&\quad\quad\times 2^{N^{\ssfu}+N^{\ssfv}+\sum\limits_{j=1}^{r^{\ssfu}} A_j^{\ssfu} +\sum\limits_{j=1}^{r^{\ssfv}} A_j^{\ssfv} - \sum\limits_{j=0}^{r^{\ssfu}} B_j^{\ssfu} - \sum\limits_{j=0}^{r^{\ssfv}} B_j^{\ssfv}} \prod_{i=1}^{r'} \frac{1}{(A'_{i}!)^2}\notag\\
	&\quad\quad\times \big(M'!\big)\prod_{n=0}^{M'-1}(L'-M'-n)\prod_{i=1}^{r'} \scalebox{0.75}{$\dfrac{1}{\sum\limits_{j=1}^i (A'_{r'+1-j}+1)}$}\prod_{i=1}^{r'} \scalebox{0.75}{$\dfrac{1}{L'-2M'+\sum\limits_{j=1}^i (A'_{r'+1-j}+1)}$},
\end{align*}
in view of \eqref{generalized-Juhl-formula}. Here the nonnegative sequences $\mathbf{A}^{\ssfu}=(A_1^{\ssfu},\ldots,A_{r^{\ssfu}}^{\ssfu})$, $\mathbf{A}^{\ssfv}=(A_1^{\ssfv},\ldots,A_{r^{\ssfv}}^{\ssfv})$ and $\mathbf{A}'=(A'_1,\ldots,A'_{r'})$ are such that
\begin{align*}
	r' + \sum_{j=1}^{r'} A'_j = M'-N^{\ssfu}-N^{\ssfv}-\sum\limits_{j=1}^{r^{\ssfu}} A_j^{\ssfu} -\sum\limits_{j=1}^{r^{\ssfv}} A_j^{\ssfv} + \sum\limits_{j=0}^{r^{\ssfu}} B_j^{\ssfu} + \sum\limits_{j=0}^{r^{\ssfv}} B_j^{\ssfv},
\end{align*}
according to \eqref{eq:generalized-Juhl-formula-cond}. This is, in light of \eqref{eq:B-cond-D2k-u} and \eqref{eq:B-cond-D2k-v}, further equivalent to
\begin{align}\label{eq:D2k-cond}
	r^{\ssfu} + \sum\limits_{j=1}^{r^{\ssfu}} A_j^{\ssfu} + r^{\ssfv} + \sum\limits_{j=1}^{r^{\ssfv}} A_j^{\ssfv} + r' + \sum_{j=1}^{r'} A'_j = M^{\ssfu} + M^{\ssfv} + M' - N^{\ssfu} - N^{\ssfv}.
\end{align}

For convenience, we write
\begin{align}
	\mM_{\mathbf{A}',\mathbf{A}^{\ssfu},\mathbf{A}^{\ssfv}}(u,v) & := \mM_{2(A'_{r'}+1)}\cdots \mM_{2(A'_1+1)}\\
	&\ \quad \big(\mM_{2(A_{r^{\ssfu}}^{\ssfu}+1)}\cdots \mM_{2(A_1^{\ssfu}+1)}(u)\mM_{2(A_{r^{\ssfv}}^{\ssfv}+1)}\cdots \mM_{2(A_1^{\ssfv}+1)}(v)\big)\notag.
\end{align}
It follows that
\begin{align}
	&D_{[M',L'],[M^{\ssfu},L^{\ssfu}],[M^{\ssfv},L^{\ssfv}]}(\mP_{N^{\ssfu}}(u)\otimes \mP_{N^{\ssfv}}(v))\\
	&=\sum_{\mathbf{A}'}\sum_{\mathbf{A}^{\ssfu}}\sum_{\mathbf{A}^{\ssfv}} \mM_{\mathbf{A}',\mathbf{A}^{\ssfu},\mathbf{A}^{\ssfv}}(u,v)\notag\\
	&\times  (-1)^{N^{\ssfu}+N^{\ssfv}} 2^{N^{\ssfu}+N^{\ssfv}} \prod_{i=1}^{r^{\ssfu}} \frac{1}{(A_{i}^{\ssfu}!)^2} \prod_{i=1}^{r^{\ssfv}} \frac{1}{(A_{i}^{\ssfv}!)^2} \prod_{i=1}^{r'} \frac{1}{(A'_{i}!)^2} \prod_{i=1}^{r'} \scalebox{0.75}{$\dfrac{1}{\sum\limits_{j=1}^i (A'_{r'+1-j}+1)}$}\notag\\
	&\times (-1)^{M'-r'} \big(M'!\big)\prod_{n=0}^{M'-1}(L'-M'-n)\prod_{i=1}^{r'} \scalebox{0.75}{$\dfrac{1}{L'-2M'+\sum\limits_{j=1}^i (A'_{r'+1-j}+1)}$}\notag\\
	&\times\sum_{\mathbf{B}^{\ssfu}}  \prod_{i=0}^{r^{\ssfu}} (-1)^{B_i^{\ssfu}} \big(B_i^{\ssfu}!\big)^2 \scalebox{0.8}{$\left(\begin{array}{c}
					N^{\ssfu}+\sum\limits_{j=1}^i A_j^{\ssfu} - \sum\limits_{j=0}^{i-1} B_j^{\ssfu}\\[12pt]
					B_i^{\ssfu}
				\end{array}\right)$}
	\scalebox{0.8}{$\left(\begin{array}{c}
					L^{\ssfu}-N^{\ssfu}-2i-\sum\limits_{j=1}^i A_j^{\ssfu} - \sum\limits_{j=0}^{i-1} B_j^{\ssfu}\\[12pt]
					B_i^{\ssfu}
				\end{array}\right)$}\notag\\
	&\times\sum_{\mathbf{B}^{\ssfv}}  \prod_{i=0}^{r^{\ssfv}} (-1)^{B_i^{\ssfv}} \big(B_i^{\ssfv}!\big)^2 \scalebox{0.8}{$\left(\begin{array}{c}
					N^{\ssfv}+\sum\limits_{j=1}^i A_j^{\ssfv} - \sum\limits_{j=0}^{i-1} B_j^{\ssfv}\\[12pt]
					B_i^{\ssfv}
				\end{array}\right)$}
	\scalebox{0.8}{$\left(\begin{array}{c}
					L^{\ssfv}-N^{\ssfv}-2i-\sum\limits_{j=1}^i A_j^{\ssfv} - \sum\limits_{j=0}^{i-1} B_j^{\ssfv}\\[12pt]
					B_i^{\ssfv}
				\end{array}\right)$},\notag
\end{align}
where $\mathbf{A}'$, $\mathbf{A}^{\ssfu}$, $\mathbf{A}^{\ssfv}$, $\mathbf{B}^{\ssfu}$ and $\mathbf{B}^{\ssfv}$ are controlled by \eqref{eq:B-cond-D2k-u}, \eqref{eq:B-cond-D2k-v} and \eqref{eq:D2k-cond}.

Now we are in a position to expand $D_{U,V,L,K^\ssfu,K^\ssfv}(\mP_{N^\ssfu}(u)\otimes \mP_{N^\ssfv}(v))$.

\begin{theorem}\label{th:D-2k-general}
	Let $N^\ssfu$ and $N^\ssfv$ be nonnegative integers. Then
	\begin{align}\label{eq:D-2k-general-N}
		&D_{U,V,L,K^\ssfu,K^\ssfv}(\mP_{N^\ssfu}(u)\otimes \mP_{N^\ssfv}(v)) \\
		&\quad = \sum_{\mathbf{A}'}\sum_{\mathbf{A}^{\ssfu}}\sum_{\mathbf{A}^{\ssfv}} \mM_{2(A'_{r'}+1)}\cdots \mM_{2(A'_1+1)}\notag\\
		&\ \quad\quad \big(\mM_{2(A_{r^{\ssfu}}^{\ssfu}+1)}\cdots \mM_{2(A_1^{\ssfu}+1)}(u)\mM_{2(A_{r^{\ssfv}}^{\ssfv}+1)}\cdots \mM_{2(A_1^{\ssfv}+1)}(v)\big)\notag\\
		&\quad\quad\times (-1)^{N^{\ssfu}+N^{\ssfv}} 2^{N^{\ssfu}+N^{\ssfv}} \frac{1}{(K^\ssfu+N^\ssfu)(K^\ssfv+N^\ssfv)(L-N^\ssfu)(L-N^\ssfv)}\notag\\
		&\quad\quad\times \prod_{n=0}^{U} (K^\ssfu+n) \prod_{n=0}^{U} (K^\ssfv+n) \prod_{n=1}^{U} (L-n) \prod_{n=0}^{V-1}(L+n)\notag\\
		&\quad\quad\times \prod_{i=1}^{r^{\ssfu}} \frac{1}{(A_{i}^{\ssfu}!)^2}\prod_{i=1}^{r^\ssfu} \scalebox{0.75}{$\dfrac{1}{K^\ssfu+N^\ssfu+\sum\limits_{j=1}^i (A_{j}^\ssfu+1)}$}\prod_{i=1}^{r^\ssfu} \scalebox{0.75}{$\dfrac{1}{L-N^\ssfu-\sum\limits_{j=1}^i (A_{j}^\ssfu+1)}$}\notag\\
		&\quad\quad\times \prod_{i=1}^{r^{\ssfv}} \frac{1}{(A_{i}^{\ssfv}!)^2}\prod_{i=1}^{r^\ssfv} \scalebox{0.75}{$\dfrac{1}{K^\ssfv+N^\ssfv+\sum\limits_{j=1}^i (A_{j}^\ssfv+1)}$}\prod_{i=1}^{r^\ssfv} \scalebox{0.75}{$\dfrac{1}{L-N^\ssfv-\sum\limits_{j=1}^i (A_{j}^\ssfv+1)}$}\notag\\
		&\quad\quad\times \prod_{i=1}^{r'} \frac{1}{(A'_{i}!)^2}\prod_{i=1}^{r'} \scalebox{0.75}{$\dfrac{1}{\sum\limits_{j=1}^i (A'_{r'+1-j}+1)}$} \prod_{i=1}^{r'} \scalebox{0.75}{$\dfrac{1}{L+V-\sum\limits_{j=1}^i (A'_{r'+1-j}+1)}$},\notag
	\end{align}
	where the summation runs over all nonnegative sequences $\mathbf{A}^{\ssfu}=(A_1^{\ssfu},\ldots,A_{r^{\ssfu}}^{\ssfu})$, $\mathbf{A}^{\ssfv}=(A_1^{\ssfv},\ldots,A_{r^{\ssfv}}^{\ssfv})$ and $\mathbf{A}'=(A'_1,\ldots,A'_{r'})$ such that
	\begin{align}\label{eq:AAA-cond}
		r^{\ssfu} + \sum\limits_{j=1}^{r^{\ssfu}} A_j^{\ssfu} + r^{\ssfv} + \sum\limits_{j=1}^{r^{\ssfv}} A_j^{\ssfv} + r' + \sum_{j=1}^{r'} A'_j = U - N^{\ssfu} - N^{\ssfv}.
	\end{align}
\end{theorem}

\begin{proof}
	We apply our previous analysis to expand each
	\begin{align*}
		D_{[M',-L-V+2M'],[M^{\ssfu},L-K^{\ssfu}],[M^{\ssfv},L-K^{\ssfv}]}(\rho^{N^\ssfu}u \otimes \rho^{N^\ssfv}v).
	\end{align*}
	First, we note that the restriction \eqref{eq:AAA-cond} comes from \eqref{eq:D2k-cond} where we have also utilized the fact that $M^{\ssfu} + M^{\ssfv} + M' = U$. Next, it is easy to observe that
	\begin{align*}
		&\frac{\Gamma(U+K^\ssfu+1)\Gamma(U+K^\ssfv+1)}{\Gamma(M^\ssfu+K^\ssfu+1)\Gamma(M^\ssfv+K^\ssfv+1)\Gamma(M'+1)}\\
		&\times \frac{\Gamma(L-M^\ssfu)\Gamma(L-M^\ssfv)\Gamma(L+V-M')}{\Gamma(L-U)\Gamma(L)^2} \\
		&\qquad\qquad = L\cdot \frac{\Gamma(L+V-M')}{\Gamma(L)}\cdot \prod_{n=0}^{U} (K^\ssfv+n) \cdot \prod_{n=0}^{U-M^\ssfu} \frac{1}{K^\ssfv+n} \prod_{n=0}^{U-M^\ssfu} \frac{1}{L-n}\\
		&\qquad\qquad\quad\times \big((U-M^\ssfu)!\big)^2 \binom{L-M^\ssfu-1}{U-M^\ssfu}\binom{U+K^\ssfu}{U-M^\ssfu}\\
		&\qquad\qquad\quad\times \big(M'!\big)\binom{L-U+M^\ssfu+M'-1}{M'}\binom{U-M^\ssfu+K^\ssfv}{M'}.
	\end{align*}
	Also,
	\begin{align*}
		\frac{\Gamma(L+V-M')}{\Gamma(L)}\prod_{n=0}^{M'-1}(-L-V+M'-n) = (-1)^{M'} \prod_{n=0}^{V-1}(L+n).
	\end{align*}
	Hence, $D_{U,V,L,K^\ssfu,K^\ssfv}(\rho^{N^\ssfu}u \otimes \rho^{N^\ssfv}v)$ equals
	\begin{align*}
		&\sum_{\mathbf{A}'}\sum_{\mathbf{A}^{\ssfu}}\sum_{\mathbf{A}^{\ssfv}} \mM_{\mathbf{A}',\mathbf{A}^{\ssfu},\mathbf{A}^{\ssfv}}(u,v)\notag\\
		&\times (-1)^{N^{\ssfu}+N^{\ssfv}} 2^{N^{\ssfu}+N^{\ssfv}} L \cdot \prod_{n=0}^{U} (K^\ssfv+n)\prod_{n=0}^{V-1}(L+n) \notag\\
		&\times \prod_{i=1}^{r^{\ssfu}} \frac{1}{(A_{i}^{\ssfu}!)^2} \prod_{i=1}^{r^{\ssfv}} \frac{1}{(A_{i}^{\ssfv}!)^2} \prod_{i=1}^{r'} \frac{1}{(A'_{i}!)^2}\prod_{i=1}^{r'} \scalebox{0.75}{$\dfrac{1}{\sum\limits_{j=1}^i (A'_{r'+1-j}+1)}$} \prod_{i=1}^{r'} \scalebox{0.75}{$\dfrac{1}{L+V-\sum\limits_{j=1}^i (A'_{r'+1-j}+1)}$}\notag\\
		&\times\sum_{M^{\ssfu}=0}^{U} \big((U-M^\ssfu)!\big)^2 \binom{L-M^\ssfu-1}{U-M^\ssfu}\binom{U+K^\ssfu}{U-M^\ssfu} \prod_{n=0}^{U-M^\ssfu} \frac{1}{K^\ssfv+n} \prod_{n=0}^{U-M^\ssfu} \frac{1}{L-n} \\
		&\times\sum_{\mathbf{B}^{\ssfu}}  \prod_{i=0}^{r^{\ssfu}} (-1)^{B_i^{\ssfu}} \big(B_i^{\ssfu}!\big)^2 \scalebox{0.7}{$\left(\begin{array}{c}
				N^{\ssfu}+\sum\limits_{j=1}^i A_j^{\ssfu} - \sum\limits_{j=0}^{i-1} B_j^{\ssfu}\\[12pt]
				B_i^{\ssfu}
			\end{array}\right)$}
		\scalebox{0.7}{$\left(\begin{array}{c}
				L-K^{\ssfu}-N^{\ssfu}-2i-\sum\limits_{j=1}^i A_j^{\ssfu} - \sum\limits_{j=0}^{i-1} B_j^{\ssfu}\\[12pt]
				B_i^{\ssfu}
			\end{array}\right)$}\notag\\
		&\times\sum_{\substack{M^{\ssfv},M'\ge 0\\M^{\ssfv}+M'=U-M^{\ssfu}}}\big(M'!\big)^2\binom{L-U+M^\ssfu+M'-1}{M'}\binom{U-M^\ssfu+K^\ssfv}{M'}\\
		&\times \sum_{\mathbf{B}^{\ssfv}}  \prod_{i=0}^{r^{\ssfv}} (-1)^{B_i^{\ssfv}} \big(B_i^{\ssfv}!\big)^2\scalebox{0.7}{$\left(\begin{array}{c}
				N^{\ssfv}+\sum\limits_{j=1}^i A_j^{\ssfv} - \sum\limits_{j=0}^{i-1} B_j^{\ssfv}\\[12pt]
				B_i^{\ssfv}
			\end{array}\right)$}
		\scalebox{0.7}{$\left(\begin{array}{c}
				L-K^{\ssfv}-N^{\ssfv}-2i-\sum\limits_{j=1}^i A_j^{\ssfv} - \sum\limits_{j=0}^{i-1} B_j^{\ssfv}\\[12pt]
				B_i^{\ssfv}
			\end{array}\right)$}.
	\end{align*}
	Now we extend the sequence $\mathbf{B}^\ssfv$ to the following sequence of length $r^{\ssfv}+2$:
	\begin{align*}
		\widehat{\mathbf{B}}^{\ssfv} = (\widehat{B}_0^{\ssfv},\widehat{B}_1^{\ssfv},\ldots,\widehat{B}_{r^\ssfv}^{\ssfv},\widehat{B}_{r^\ssfv+1}^{\ssfv}) \mapsto (B_0^{\ssfv},B_1^{\ssfv},\ldots,B_{r^\ssfv}^{\ssfv},M').
	\end{align*}
	It is clear that
	\begin{align*}
		\sum_{j=0}^{r^\ssfv+1} \widehat{B}_j^{\ssfv} &= \sum_{j=0}^{r^\ssfv} B_j^{\ssfv} + M' = M^\ssfv + M' - r^\ssfv\\
		&= (U-M^\ssfu+1) - (r^\ssfv + 1),
	\end{align*}
	where we have applied \eqref{eq:B-cond-D2k-v} for the second equality. In \eqref{eq:aux-sum-2}, we replace $\mathbf{C}$ with $\widehat{\mathbf{B}}^{\ssfv}$ and make the following substitutions:
	\begin{align*}
		\mathbf{A} &\mapsto (N^\ssfv,A_1^\ssfv,\ldots,A_{r^\ssfv}^\ssfv),\\
		r &\mapsto r^\ssfv+1,\\
		M &\mapsto U-M^\ssfu+1,\\
		X &\mapsto L-U+M^\ssfu,\\
		Y &\mapsto -K^\ssfv+1.
	\end{align*}
	It follows that
	\begin{align*}
		&\sum_{\substack{M^{\ssfv},M'\ge 0\\M^{\ssfv}+M'=U-M^{\ssfu}}} \big(M'!\big)^2 \binom{L-U+M^\ssfu+M'-1}{M'}\binom{U-M^\ssfu+K^\ssfv}{M'} \\
		&\times\sum_{\mathbf{B}^{\ssfv}}  \prod_{i=0}^{r^{\ssfv}} (-1)^{B_i^{\ssfv}} \big(B_i^{\ssfv}!\big)^2\scalebox{0.7}{$\left(\begin{array}{c}
				N^{\ssfv}+\sum\limits_{j=1}^i A_j^{\ssfv} - \sum\limits_{j=0}^{i-1} B_j^{\ssfv}\\[12pt]
				B_i^{\ssfv}
			\end{array}\right)$}
		\scalebox{0.7}{$\left(\begin{array}{c}
				L-K^{\ssfv}-N^{\ssfv}-2i-\sum\limits_{j=1}^i A_j^{\ssfv} - \sum\limits_{j=0}^{i-1} B_j^{\ssfv}\\[12pt]
				B_i^{\ssfv}
			\end{array}\right)$}
	\end{align*}
	equals
	\begin{align*}
		&\frac{1}{(K^\ssfv+N^\ssfv)(L-N^\ssfv)} \prod_{n=0}^{U-M^\ssfu} (K^\ssfv+n)\prod_{n=0}^{U-M^\ssfu} (L-n)\\
		&\times \prod_{i=1}^{r^\ssfv} \scalebox{0.75}{$\dfrac{1}{K^\ssfv+N^\ssfv+\sum\limits_{j=1}^i (A_{j}^\ssfv+1)}$}\prod_{i=1}^{r^\ssfv} \scalebox{0.75}{$\dfrac{1}{L-N^\ssfv-\sum\limits_{j=1}^i (A_{j}^\ssfv+1)}$}.
	\end{align*}
	Therefore, $D_{U,V,L,K^\ssfu,K^\ssfv}(\rho^{N^\ssfu}u \otimes \rho^{N^\ssfv}v)$ further equals
	\begin{align*}
		&\sum_{\mathbf{A}'} \sum_{\mathbf{A}^{\ssfu}}\sum_{\mathbf{A}^{\ssfv}} \mM_{\mathbf{A}',\mathbf{A}^{\ssfu},\mathbf{A}^{\ssfv}}(u,v)\notag\\
		&\times (-1)^{N^{\ssfu}+N^{\ssfv}} 2^{N^{\ssfu}+N^{\ssfv}} \frac{L}{(K^\ssfv+N^\ssfv)(L-N^\ssfv)}\prod_{n=0}^{U} (K^\ssfv+n) \prod_{n=0}^{V-1}(L+n) \\
		&\times \prod_{i=1}^{r^{\ssfu}} \frac{1}{(A_{i}^{\ssfu}!)^2} \prod_{i=1}^{r^{\ssfv}} \frac{1}{(A_{i}^{\ssfv}!)^2} \prod_{i=1}^{r'} \frac{1}{(A'_{i}!)^2} \notag\\
		&\times \prod_{i=1}^{r'} \scalebox{0.75}{$\dfrac{1}{\sum\limits_{j=1}^i (A'_{r'+1-j}+1)}$} \prod_{i=1}^{r'} \scalebox{0.75}{$\dfrac{1}{L+V-\sum\limits_{j=1}^i (A'_{r'+1-j}+1)}$}\notag\\
		&\times \prod_{i=1}^{r^\ssfv} \scalebox{0.75}{$\dfrac{1}{K^\ssfv+N^\ssfv+\sum\limits_{j=1}^i (A_{j}^\ssfv+1)}$}\prod_{i=1}^{r^\ssfv} \scalebox{0.75}{$\dfrac{1}{L-N^\ssfv-\sum\limits_{j=1}^i (A_{j}^\ssfv+1)}$}\\
		&\times\sum_{M^{\ssfu}=0}^{U} \big((U-M^\ssfu)!\big)^2 \binom{L-M^\ssfu-1}{U-M^\ssfu}\binom{U+K^\ssfu}{U-M^\ssfu} \\
		&\times\sum_{\mathbf{B}^{\ssfu}}  \prod_{i=0}^{r^{\ssfu}} (-1)^{B_i^{\ssfu}} \big(B_i^{\ssfu}!\big)^2 \scalebox{0.7}{$\left(\begin{array}{c}
				N^{\ssfu}+\sum\limits_{j=1}^i A_j^{\ssfu} - \sum\limits_{j=0}^{i-1} B_j^{\ssfu}\\[12pt]
				B_i^{\ssfu}
			\end{array}\right)$}
		\scalebox{0.7}{$\left(\begin{array}{c}
				L-K^{\ssfu}-N^{\ssfu}-2i-\sum\limits_{j=1}^i A_j^{\ssfu} - \sum\limits_{j=0}^{i-1} B_j^{\ssfu}\\[12pt]
				B_i^{\ssfu}
			\end{array}\right)$}.
	\end{align*}
	This time we extend the sequence $\mathbf{B}^\ssfu$ to the following sequence of length $r^{\ssfu}+2$:
	\begin{align*}
		\widehat{\mathbf{B}}^{\ssfu} = (\widehat{B}_0^{\ssfu},\widehat{B}_1^{\ssfu},\ldots,\widehat{B}_{r^\ssfu}^{\ssfu},\widehat{B}_{r^\ssfu+1}^{\ssfu}) \mapsto (B_0^{\ssfu},B_1^{\ssfu},\ldots,B_{r^\ssfu}^{\ssfu},U-M^\ssfu).
	\end{align*}
	Then by \eqref{eq:B-cond-D2k-u},
	\begin{align*}
		\sum_{j=0}^{r^\ssfu+1} \widehat{B}_j^{\ssfu} &= \sum_{j=0}^{r^\ssfu} B_j^{\ssfu} + (U-M^\ssfu) = M^\ssfu + (U-M^\ssfu) - r^\ssfu\\
		&= (U+1) - (r^\ssfu + 1).
	\end{align*}
	Let us replace $\mathbf{C}$ with $\widehat{\mathbf{B}}^{\ssfu}$ in \eqref{eq:aux-sum-2} and make the following substitutions:
	\begin{align*}
		\mathbf{A} &\mapsto (N^\ssfu,A_1^\ssfu,\ldots,A_{r^\ssfu}^\ssfu),\\
		r &\mapsto r^\ssfu+1,\\
		M &\mapsto U+1,\\
		X &\mapsto L-U,\\
		Y &\mapsto -K^\ssfu+1.
	\end{align*}
	Thus,
	\begin{align*}
		&\sum_{M^{\ssfu}=0}^{U} \big((U-M^\ssfu)!\big)^2 \binom{L-M^\ssfu-1}{U-M^\ssfu}\binom{U+K^\ssfu}{U-M^\ssfu} \\
		&\times \sum_{\mathbf{B}^{\ssfu}} \prod_{i=0}^{r^{\ssfu}} (-1)^{B_i^{\ssfu}} \big(B_i^{\ssfu}!\big)^2 \scalebox{0.7}{$\left(\begin{array}{c}
				N^{\ssfu}+\sum\limits_{j=1}^i A_j^{\ssfu} - \sum\limits_{j=0}^{i-1} B_j^{\ssfu}\\[12pt]
				B_i^{\ssfu}
			\end{array}\right)$}
		\scalebox{0.7}{$\left(\begin{array}{c}
				L-K^{\ssfu}-N^{\ssfu}-2i-\sum\limits_{j=1}^i A_j^{\ssfu} - \sum\limits_{j=0}^{i-1} B_j^{\ssfu}\\[12pt]
				B_i^{\ssfu}
			\end{array}\right)$}
	\end{align*}
	equals
	\begin{align*}
		&\frac{1}{(K^\ssfu+N^\ssfu)(L-N^\ssfu)} \prod_{n=0}^{U} (K^\ssfu+n)\prod_{n=0}^{U} (L-n)\\
		&\times \prod_{i=1}^{r^\ssfu} \scalebox{0.75}{$\dfrac{1}{K^\ssfu+N^\ssfu+\sum\limits_{j=1}^i (A_{j}^\ssfu+1)}$}\prod_{i=1}^{r^\ssfu} \scalebox{0.75}{$\dfrac{1}{L-N^\ssfu-\sum\limits_{j=1}^i (A_{j}^\ssfu+1)}$}.
	\end{align*}
	Substituting this relation into the expression of $D_{U,V,L,K^\ssfu,K^\ssfv}(\rho^{N^\ssfu}u \otimes \rho^{N^\ssfv}v)$ derived earlier yields the desired result.
\end{proof}

We are interested in for which choice of parameters the operator
\begin{align*}
	D_{U,V,L,K^\ssfu,K^\ssfv}(\rho^{N^\ssfu}u \otimes \rho^{N^\ssfv}v)
\end{align*}
is formally self-adjoint. Let $\mC_{(\mathbf{A}', \mathbf{A}^\ssfu,\mathbf{A}^\ssfv)}$ be the coefficient of the monomial
\begin{align*}
	&\mM_{2(A'_{r'}+1)}\cdots \mM_{2(A'_1+1)}\notag\\
	&\qquad\big(\mM_{2(A_{r^{\ssfu}}^{\ssfu}+1)}\cdots \mM_{2(A_1^{\ssfu}+1)}(u)\mM_{2(A_{r^{\ssfv}}^{\ssfv}+1)}\cdots \mM_{2(A_1^{\ssfv}+1)}(v)\big).
\end{align*}
Then these coefficients should satisfy
\begin{align*}
    \mC_{(\mathbf{A}', \mathbf{A}^\ssfu,\mathbf{A}^\ssfv)} = \mC_{(\mathbf{A}', \mathbf{A}^\ssfv, \mathbf{A}^\ssfu)}&=\mC_{((\mathbf{A}^\ssfu)^{-1}, (\mathbf{A}')^{-1}, \mathbf{A}^\ssfv)}=\mC_{((\mathbf{A}^\ssfv)^{-1}, (\mathbf{A}')^{-1}, \mathbf{A}^\ssfu)}\\
    &= \mC_{((\mathbf{A}^\ssfu)^{-1}, \mathbf{A}^\ssfv, (\mathbf{A}')^{-1})}=\mC_{((\mathbf{A}^\ssfv)^{-1}, \mathbf{A}^\ssfu, (\mathbf{A}')^{-1})}.
\end{align*}
Recalling also that $N^\ssfu$, $N^\ssfv$ and $V$ are nonnegative integers, the above imply that $K^\ssfu=K^\ssfv=0$, $N^\ssfu=N^\ssfv=0$ and $V=0$. 

\begin{corollary}\label{formally-self-adjointness-2}
We have
	\begin{align}\label{eq:D-2k-general-uv}
		&D_{U,0,L,0,0}(u\otimes v) \\
		&\quad = \sum_{\mathbf{A}'}\sum_{\mathbf{A}^{\ssfu}}\sum_{\mathbf{A}^{\ssfv}} \mM_{2(A'_{r'}+1)}\cdots \mM_{2(A'_1+1)}\notag\\
		&\ \quad\quad \big(\mM_{2(A_{r^{\ssfu}}^{\ssfu}+1)}\cdots \mM_{2(A_1^{\ssfu}+1)}(u)\mM_{2(A_{r^{\ssfv}}^{\ssfv}+1)}\cdots \mM_{2(A_1^{\ssfv}+1)}(v)\big)\notag\\
		&\quad\quad\times \frac{(U!)^2}{L^2}\prod_{n=1}^U (L-n)\prod_{i=1}^{r^{\ssfu}} \frac{1}{(A_{i}^{\ssfu}!)^2} \prod_{i=1}^{r^{\ssfv}} \frac{1}{(A_{i}^{\ssfv}!)^2} \prod_{i=1}^{r'} \frac{1}{(A'_{i}!)^2}\notag\\
		&\quad\quad\times \prod_{i=1}^{r^\ssfu} \scalebox{0.75}{$\dfrac{1}{\sum\limits_{j=1}^i (A_{j}^\ssfu+1)}$}\prod_{i=1}^{r^\ssfu} \scalebox{0.75}{$\dfrac{1}{L-\sum\limits_{j=1}^i (A_{j}^\ssfu+1)}$} \prod_{i=1}^{r^\ssfv} \scalebox{0.75}{$\dfrac{1}{\sum\limits_{j=1}^i (A_{j}^\ssfv+1)}$}\prod_{i=1}^{r^\ssfv} \scalebox{0.75}{$\dfrac{1}{L-\sum\limits_{j=1}^i (A_{j}^\ssfv+1)}$}\notag\\
		&\quad\quad\times \prod_{i=1}^{r'} \scalebox{0.75}{$\dfrac{1}{\sum\limits_{j=1}^i (A'_{r'+1-j}+1)}$} \prod_{i=1}^{r'} \scalebox{0.75}{$\dfrac{1}{L-\sum\limits_{j=1}^i (A'_{r'+1-j}+1)}$},\notag
	\end{align}
	where the summation runs over all nonnegative sequences $\mathbf{A}^{\ssfu}=(A_1^{\ssfu},\ldots,A_{r^{\ssfu}}^{\ssfu})$, $\mathbf{A}^{\ssfv}=(A_1^{\ssfv},\ldots,A_{r^{\ssfv}}^{\ssfv})$ and $\mathbf{A}'=(A'_1,\ldots,A'_{r'})$ such that
	\begin{align*}
		r^{\ssfu} + \sum\limits_{j=1}^{r^{\ssfu}} A_j^{\ssfu} + r^{\ssfv} + \sum\limits_{j=1}^{r^{\ssfv}} A_j^{\ssfv} + r' + \sum_{j=1}^{r'} A'_j = U.
	\end{align*}
	In particular, $D_{U,0,L,0,0}$ is formally self-adjoint.
\end{corollary}

Finally, we recall that
\begin{align*}
	D_{2k}(u\otimes v) = \frac{1}{k!}D_{k,0,L_k,0,0}(u\otimes v),
\end{align*}
and hence complete the proof of Theorem \ref{JUhl-formula-ovsienko-redou}.
\section{Formal self-adjointness of $D_{2k;\mathcal{I}}$}\label{sec:selfadj-1}

According to \eqref{eq:D-2k-I}, the operator $D_{2k;\mathcal{I}}$ is a specialization of 
\begin{align*}
	D_{U,V,L,K;f}(u) &
	:= \sum_{\substack{M,M'\ge 0\\M+M'=U}} \binom{U+K}{M'}\frac{\Gamma(L+M')\Gamma(L+V-M')}{\Gamma(L)^2}\\
	&\ \quad\times D_{[M',-L-V+2M'],[M,L-K+U];f}(u)
\end{align*}
by taking $U=V=k$, $L=\ell$, $K=0$ and replacing $f$ with $\cI$.

In view of this, the main objective of this section is to derive an explicit expansion of the operator $D_{U,V,L,K;f}$. To achieve this goal, we first need to look into the operators
\begin{align*}
	D_{[M',L'],[M,L];f}(u) :=  \cD_{M',L'}\big(f\,\cD_{M,L}(u)\big)\Big\vert_{\rho=0}.
\end{align*}
Recall that $u$ is a smooth function independent of $\rho$ and that $N$ is a nonnegative integer.

For the inner layer, we may use \eqref{eq:S-prefactor} and \eqref{eq:S(A,B)} to get
\begin{align*}
	&\mR_{L+1-2M} \cdots \mR_{L-3} \mR_{L-1} (\rho^N u)\\
	&\quad = \sum_{\mathbf{A}}\mM_{2(A_r+1)}\cdots \mM_{2(A_1+1)}(u) \cdot (-1)^{\sum\limits_{j=1}^r A_j} 2^{-\sum\limits_{j=1}^r A_j + \sum\limits_{j=0}^r B_j} \prod_{i=1}^r \frac{1}{(A_{i}!)^2}\\
	&\quad\quad\times\sum_{\mathbf{B}}  \prod_{i=0}^{r} \big(B_i!\big)^2 \scalebox{0.8}{$\left(\begin{array}{c}
			N+\sum\limits_{j=1}^i A_j - \sum\limits_{j=0}^{i-1} B_j\\[12pt]
			B_i
		\end{array}\right)$}
	\scalebox{0.8}{$\left(\begin{array}{c}
			L-N-2i-\sum\limits_{j=1}^i A_j - \sum\limits_{j=0}^{i-1} B_j\\[12pt]
			B_i
		\end{array}\right)$}\\
	&\quad\quad\times \rho^{N+\sum\limits_{j=1}^r A_j - \sum\limits_{j=0}^r B_j}.
\end{align*}
Here the nonnegative sequence $\mathbf{B}=(B_0,B_1,\ldots,B_r)$ is such that
\begin{align}\label{eq:B-cond-3}
	r + \sum_{j=0}^r B_j = M,
\end{align}
which comes from \eqref{eq:B-cond}.

For the outer layer, we are essentially looking at
\begin{align*}
	&\mR_{L'+1-2M'} \cdots \mR_{L'-3} \mR_{L'-1}\left(\rho^{N+\sum\limits_{j=1}^r A_j - \sum\limits_{j=0}^r B_j} f\right)\Bigg|_{\rho=0}\\
	&\quad= \sum_{R}\sum_{\mathbf{A}'} \mM_{2(A'_{r'}+1)}\cdots \mM_{2(A'_1+1)} (f^{(R)})\\
	&\quad\quad\times  (-1)^{M'-N-\sum\limits_{j=1}^r A_j + \sum\limits_{j=0}^r B_j-R-r'}\, 2^{N+\sum\limits_{j=1}^r A_j - \sum\limits_{j=0}^r B_j+R} \frac{1}{R!}\prod_{i=1}^{r'} \frac{1}{(A'_{i}!)^2}\notag\\
	&\quad\quad\times \big(M'!\big)\prod_{n=0}^{M'-1}(L'-M'-n)\prod_{i=1}^{r'} \scalebox{0.75}{$\dfrac{1}{\sum\limits_{j=1}^i (A'_{r'+1-j}+1)}$}\prod_{i=1}^{r'} \scalebox{0.75}{$\dfrac{1}{L'-2M'+\sum\limits_{j=1}^i (A'_{r'+1-j}+1)}$},
\end{align*}
in view of \eqref{eq:D-f-u}. Here the nonnegative integer $R$ and nonnegative sequences $\mathbf{A}=(A_1,\ldots,A_r)$ and $\mathbf{A}'=(A'_1,\ldots,A'_{r'})$ are such that
\begin{align*}
	R + r' + \sum_{j=1}^{r'} A'_j = M'-N-\sum\limits_{j=1}^r A_j + \sum\limits_{j=0}^r B_j,
\end{align*}
according to \eqref{eq:R-cond}. This is, in light of \eqref{eq:B-cond-3}, further equivalent to
\begin{align}\label{eq:R-cond-3}
	R + r + \sum\limits_{j=1}^r A_j + r' + \sum_{j=1}^{r'} A'_j = M + M' - N.
\end{align}

It follows from the above discussion that
\begin{align}\label{eq:D-f-level-2}
	&D_{[M',L'],[M,L];f}\circ \mP_N( u)\\
	&= \sum_{R}\sum_{\mathbf{A}'}\sum_{\mathbf{A}} \mM_{2(A'_{r'}+1)}\cdots \mM_{2(A'_1+1)} \big(f^{(R)}\mM_{2(A_r+1)}\cdots \mM_{2(A_1+1)}(u)\big)\notag\\
	&\quad\times  (-1)^{N+R}\, 2^{N+R} \frac{1}{R!}\prod_{i=1}^r \frac{1}{(A_{i}!)^2}\prod_{i=1}^{r'} \frac{1}{(A'_{i}!)^2}\prod_{i=1}^{r'} \scalebox{0.75}{$\dfrac{1}{\sum\limits_{j=1}^i (A'_{r'+1-j}+1)}$}\notag\\
	&\quad\times (-1)^{M'-r'}\big(M'!\big)\prod_{n=0}^{M'-1}(L'-M'-n)\prod_{i=1}^{r'} \scalebox{0.75}{$\dfrac{1}{L'-2M'+\sum\limits_{j=1}^i (A'_{r'+1-j}+1)}$}\notag\\
	&\quad\times\sum_{\mathbf{B}}  \prod_{i=0}^{r} (-1)^{B_i} \big(B_i!\big)^2 \scalebox{0.8}{$\left(\begin{array}{c}
					N+\sum\limits_{j=1}^i A_j - \sum\limits_{j=0}^{i-1} B_j\\[12pt]
					B_i
				\end{array}\right)$}
	\scalebox{0.8}{$\left(\begin{array}{c}
					L-N-2i-\sum\limits_{j=1}^i A_j - \sum\limits_{j=0}^{i-1} B_j\\[12pt]
					B_i
				\end{array}\right)$},\notag
\end{align}
where $R$, $\mathbf{A}'$, $\mathbf{A}$ and $\mathbf{B}$ are controlled by \eqref{eq:B-cond-3} and \eqref{eq:R-cond-3}.

Now we are ready to produce an explicit expression of $D_{U,V,L,K;f}(\rho^N u)$.

\begin{theorem}\label{generalized-Juhl-linear-formula}
	Let $N$ be a nonnegative integer. Then
	\begin{align}\label{eq:D-general-N}
		&D_{U,V,L,K;f}\circ \mP_N( u) \\
		&\quad = \sum_{R}\sum_{\mathbf{A}'}\sum_{\mathbf{A}} \mM_{2(A'_{r'}+1)}\cdots \mM_{2(A'_1+1)} \big(f^{(R)}\mM_{2(A_r+1)}\cdots \mM_{2(A_1+1)}(u)\big)\notag\\
		&\quad\quad\times (-1)^{N+R}\, 2^{N+R} \frac{1}{R!} \frac{1}{(K+N)(L+U-N)}\notag\\
		&\quad\quad\times \prod_{n=0}^U (K+n) \prod_{n=0}^U (L+n) \prod_{n=0}^{V-1}(L+n)\notag\\
		&\quad\quad\times \prod_{i=1}^r \frac{1}{(A_{i}!)^2}\prod_{i=1}^{r} \scalebox{0.75}{$\dfrac{1}{K+N+\sum\limits_{j=1}^i (A_{j}+1)}$}\prod_{i=1}^{r} \scalebox{0.75}{$\dfrac{1}{L+U-N-\sum\limits_{j=1}^i (A_{j}+1)}$}\notag\\
		&\quad\quad\times \prod_{i=1}^{r'} \frac{1}{(A'_{i}!)^2}\prod_{i=1}^{r'} \scalebox{0.75}{$\dfrac{1}{\sum\limits_{j=1}^i (A'_{r'+1-j}+1)}$}\prod_{i=1}^{r'} \scalebox{0.75}{$\dfrac{1}{L+V-\sum\limits_{j=1}^i (A'_{r'+1-j}+1)}$}, \notag
	\end{align}
	where the summation runs over all nonnegative integers $R$ and all sequences $\mathbf{A}=(A_{1},\ldots,A_{r})$ and $\mathbf{A}'=(A'_1,\ldots,A'_{r'})$ of nonnegative integers such that
	\begin{align}\label{eq:AAR-cond}
		R + r + \sum\limits_{j=1}^r A_j + r' + \sum_{j=1}^{r'} A'_j  = U - N.
	\end{align}
	Here by abuse of notation, we read $f^{(R)}$ as $f^{(R)}\big|_{\rho=0}$.
\end{theorem}

\begin{proof}
	We directly use the previous analysis to evaluate each
	\begin{align*}
		D_{[M',-L-V+2M'],[M,L-K+U];f}(\rho^N u).
	\end{align*}
	Firstly, the restriction \eqref{eq:AAR-cond} simply comes from \eqref{eq:R-cond-3}. Meanwhile, by \eqref{eq:D-f-level-2}, we know that
	\begin{align*}
		\binom{U+K}{M'}\frac{\Gamma(L+M')\Gamma(L+V-M')}{\Gamma(L)^2}D_{[M',-L-V+2M'],[M,L-K+U];f}(\rho^N u)
	\end{align*}
	equals
	\begin{align*}
		&\sum_{R}\sum_{\mathbf{A}'}\sum_{\mathbf{A}} \mM_{2(A'_{r'}+1)}\cdots \mM_{2(A'_1+1)} \big(f^{(R)}\mM_{2(A_r+1)}\cdots \mM_{2(A_1+1)}(u)\big)\notag\\
		&\times  (-1)^{N+R}\, 2^{N+R} \frac{1}{R!} \prod_{n=0}^{V-1}(L+n)\prod_{i=1}^r \frac{1}{(A_{i}!)^2}\prod_{i=1}^{r'} \frac{1}{(A'_{i}!)^2}\\
		&\times\prod_{i=1}^{r'} \scalebox{0.75}{$\dfrac{1}{\sum\limits_{j=1}^i (A'_{r'+1-j}+1)}$}\prod_{i=1}^{r'} \scalebox{0.75}{$\dfrac{1}{L+V-\sum\limits_{j=1}^i (A'_{r'+1-j}+1)}$}\cdot \sum_{\mathbf{B}} \Pi_{M',\mathbf{B}},
	\end{align*}
	where
	\begin{align*}
		\Pi_{M',\mathbf{B}}&:= \big(M'!\big)^2 \binom{L+M'-1}{M'}\binom{U+K}{M'}\prod_{i=0}^{r} (-1)^{B_i} \big(B_i!\big)^2\\
		&\ \quad\times \scalebox{0.8}{$\left(\begin{array}{c}
				N+\sum\limits_{j=1}^i A_j - \sum\limits_{j=0}^{i-1} B_j\\[12pt]
				B_i
			\end{array}\right)$}
		\scalebox{0.8}{$\left(\begin{array}{c}
				L-K+U-N-2i-\sum\limits_{j=1}^i A_j - \sum\limits_{j=0}^{i-1} B_j\\[12pt]
				B_i
			\end{array}\right)$}.
	\end{align*}
	In view of the outer summation on $M'$, we are left to analyze
	$$\sum_{M'}\sum_{\mathbf{B}} \Pi_{M',\mathbf{B}}.$$
	Recall from \eqref{eq:B-cond-3} that $\mathbf{B}=(B_0,B_1,\ldots,B_r)$ is such that
	\begin{align*}
		r + \sum_{j=0}^r B_j = M.
	\end{align*}
	Hence, we extend $\mathbf{B}$ to a new sequence $\widehat{\mathbf{B}}$ of length $r+2$:
	\begin{align*}
		\widehat{\mathbf{B}} = (\widehat{B}_0,\widehat{B}_1\ldots, \widehat{B}_{r}, \widehat{B}_{r+1}) \mapsto (B_0,B_1,\ldots,B_r,M').
	\end{align*}
	In particular,
	\begin{align*}
		\sum_{j=0}^{r+1} \widehat{B}_j &= \sum_{j=0}^r B_j + M' = M + M' - r \\
		&= (U+1) - (r+1).
	\end{align*}
	Writing the summation $\sum_{M'}\sum_{\mathbf{B}} \Pi_{M',\mathbf{B}}$ in terms of $\widehat{\mathbf{B}}$, we see that
	\begin{align*}
		&\sum_{\widehat{\mathbf{B}}} \big(\widehat{B}_{r+1}!\big)^2 \binom{L+\widehat{B}_{r+1}-1}{\widehat{B}_{r+1}}\binom{U+K}{\widehat{B}_{r+1}}\prod_{i=0}^{r} (-1)^{\widehat{B}_i} \big(\widehat{B}_i!\big)^2\\
		&\times \scalebox{0.8}{$\left(\begin{array}{c}
				N+\sum\limits_{j=1}^i A_j - \sum\limits_{j=0}^{i-1} \widehat{B}_j\\[12pt]
				\widehat{B}_i
			\end{array}\right)$}
		\scalebox{0.8}{$\left(\begin{array}{c}
				L-K+U-N-2i-\sum\limits_{j=1}^i A_j - \sum\limits_{j=0}^{i-1} \widehat{B}_j\\[12pt]
				\widehat{B}_i
			\end{array}\right)$}
	\end{align*}
	equals
	\begin{align*}
		&\frac{1}{(K+N)(L+U-N)} \prod_{n=0}^U (K+n) \prod_{n=0}^U (L+n)\\
		&\times \prod_{i=1}^{r} \scalebox{0.75}{$\dfrac{1}{K+N+\sum\limits_{j=1}^i (A_{j}+1)}$}\prod_{i=1}^{r} \scalebox{0.75}{$\dfrac{1}{L+U-N-\sum\limits_{j=1}^i (A_{j}+1)}$},
	\end{align*}
	where we have applied \eqref{eq:aux-sum-2} with the substitutions:
	\begin{align*}
		\mathbf{A} &\mapsto (N,A_1,\ldots,A_r),\\
		r &\mapsto r+1,\\
		M &\mapsto U+1,\\
		X &\mapsto L,\\
		Y &\mapsto -K+1.
	\end{align*}
	Thus, our proof is complete.
\end{proof}

Now we consider the case that $D_{U,V,L,K;f}\circ \mP_N$ is formally self-adjoint. In this circumstance, the coefficient $\mC_{(\mathbf{A}', \mathbf{A})}$ of each monomial $$\mM_{2(A'_{r'}+1)}\cdots \mM_{2(A'_1+1)} \big(f^{(R)}\mM_{2(A_r+1)}\cdots \mM_{2(A_1+1)}(u)\big)$$ should satisfy $\mC_{(\mathbf{A}', \mathbf{A})}=\mC_{(\mathbf{A}^{-1}, (\mathbf{A}')^{-1})}$, which implies that $K=0$, $N=0$ and $V=U$. 

\begin{corollary}\label{formal-self-adjointness}
	We have
	\begin{align}
		&D_{U,U,L,0;f}(u) \\
		&\quad = \sum_{R}\sum_{\mathbf{A}'}\sum_{\mathbf{A}} \mM_{2(A'_{r'}+1)}\cdots \mM_{2(A'_1+1)} \big(f^{(R)}\mM_{2(A_r+1)}\cdots \mM_{2(A_1+1)}(u)\big)\notag\\
		&\quad\quad\times (-1)^{R}\, 2^{R}\, \frac{U!}{R!}\, \prod_{n=0}^{U-1} (L+n)^2 \prod_{i=1}^r \frac{1}{(A_{i}!)^2} \prod_{i=1}^{r'} \frac{1}{(A'_{i}!)^2}\notag\\
		&\quad\quad\times \prod_{i=1}^{r} \scalebox{0.75}{$\dfrac{1}{\sum\limits_{j=1}^i (A_{j}+1)}$}\prod_{i=1}^{r} \scalebox{0.75}{$\dfrac{1}{L+U-\sum\limits_{j=1}^i (A_{j}+1)}$} \notag\\
		&\quad\quad\times \prod_{i=1}^{r'} \scalebox{0.75}{$\dfrac{1}{\sum\limits_{j=1}^i (A'_{r'+1-j}+1)}$}\prod_{i=1}^{r'} \scalebox{0.75}{$\dfrac{1}{L+U-\sum\limits_{j=1}^i (A'_{r'+1-j}+1)}$},\notag
	\end{align}
	where the summation runs over all nonnegative integers $R$ and all sequences $\mathbf{A}=(A_{1},\ldots,A_{r})$ and $\mathbf{A}'=(A'_1,\ldots,A'_{r'})$ of nonnegative integers such that
	\begin{align*}
		R + r + \sum\limits_{j=1}^r A_j + r' + \sum_{j=1}^{r'} A'_j  = U.
	\end{align*}
	In particular, $D_{U,U,L,0;f}$ is formally self-adjoint.
\end{corollary}

The above result matches the shape of the operators $D_{2k;\mI}$ according to our earlier discussion:
\begin{align*}
	D_{2k;\mI}(u) = D_{k,k,\ell,0;\cI}(u),
\end{align*}
thereby closing the proof of Theorem \ref{JUhl-formula-linear}.
\appendix
\section{Combinatorial identities}\label{sec:appendix}

In this Appendix, we prove two combinatorial identities required for our arguments in the main context, and these identities are closely related to the evaluation of hypergeometric series. To begin with, we recall that the \emph{Pochhammer symbol} is defined for $n\in\mathbb{N}\cup\{\infty\}$,
\begin{align*}
	(a)_n := \prod_{k=0}^{n-1} (a+k),
\end{align*}
and that the \emph{hypergeometric series} is defined by
\begin{align*}
	{}_{r}F_{s} \left(\begin{matrix}
		a_1,\ldots,a_r\\
		b_1,\ldots,b_s
	\end{matrix};z\right) := \sum_{n\ge 0}\frac{(a_1)_n\cdots (a_r)_n}{(b_1)_n \cdots (b_s)_n}\frac{z^n}{n!}.
\end{align*}

We need the \emph{Pfaff--Saalsch\"utz summation formula}~\cite[p.~69, eq.~(2.2.8)]{AndrewsAskeyRoy1999}.

\begin{lemma}[Pfaff--Saalsch\"utz]
	\begin{align}\label{eq:P-S}
		{}_{3}F_{2}\left(\begin{matrix}
			-n,a,b\\
			c,1+a+b-c-n
		\end{matrix};1\right) = \frac{(c-a)_n (c-b)_n}{(c)_n (c-a-b)_n}.
	\end{align}
\end{lemma}

The first combinatorial identity is as follows.

\begin{lemma}
	Let $\mathbf{A}=(A_{1},\ldots,A_{r})$ be a sequence (which may be empty) of nonnegative integers. Write $N=\sum_{i=1}^r (A_i+1)$. Then for any indeterminate $X$,
	\begin{align}\label{eq:aux-sum-1}
		&\sum_{\mathbf{B}} \prod_{i=1}^{r} \big(B_i!\big)^2 \scalebox{0.8}{$\left(\begin{array}{c}
			\sum\limits_{j=1}^i A_j - \sum\limits_{j=1}^{i-1} B_j\\[12pt]
			B_i
		\end{array}\right)$}
		\scalebox{0.8}{$\left(\begin{array}{c}
			X-2i-\sum\limits_{j=1}^i A_j - \sum\limits_{j=1}^{i-1} B_j\\[12pt]
			B_i
		\end{array}\right)$}\\
		&\qquad = \big(N!\big) \prod_{i=1}^r \scalebox{0.75}{$\dfrac{1}{\sum\limits_{j=1}^i (A_{r+1-j}+1)}$} \prod_{n=0}^{N-1}(X-N-n) \prod_{i=1}^r \scalebox{0.75}{$\dfrac{1}{X-2N+\sum\limits_{j=1}^i (A_{r+1-j}+1)}$},\notag
	\end{align}
	where the summation runs over sequences $\mathbf{B}=(B_1,\ldots,B_r)$ of nonnegative integers such that
	\begin{align*}
		\sum_{i=1}^\ell B_i = \sum_{i=1}^r A_i = N-r.
	\end{align*}
\end{lemma}

\begin{proof}
	We prove our result by induction on $r$, the length of the sequence $\mathbf{A}$. If $r=0$, then $\mathbf{A}$ is the empty sequence, and therefore the only choice of $\mathbf{B}$ is also the empty sequence. It follows that in this case, both sides of \eqref{eq:aux-sum-1} are identical to one.
	
	Now we assume that the relation is true for all sequences $\mathbf{A}$ of length $0,\ldots,r-1$ for a certain $r\ge 1$, and we are going to prove it for an arbitrary sequence $\mathbf{A}=(A_{1},\ldots,A_{r})$. To do so, we single out the summation on $B_1$ and get
	\begin{align*}
		\LHS\eqref{eq:aux-sum-1} &= \sum_{B_1} \big(B_1!\big)^2 \binom{A_1}{B_1} \binom{X-2-A_1}{B_1}\\
		&\quad\times \sum_{B_2,\ldots,B_r} \prod_{i=2}^{r} \big(B_i!\big)^2 \scalebox{0.8}{$\left(\begin{array}{c}
			\sum\limits_{j=1}^i A_j - \sum\limits_{j=1}^{i-1} B_j\\[12pt]
			B_i
		\end{array}\right)$}
		\scalebox{0.8}{$\left(\begin{array}{c}
			X-2i-\sum\limits_{j=1}^i A_j - \sum\limits_{j=1}^{i-1} B_j\\[12pt]
			B_i
		\end{array}\right)$}.
	\end{align*}
	For the inner summation, we apply the inductive assumption for the following sequence of length $r-1$:
	\begin{align*}
		\widehat{\mathbf{A}} = (A_1+A_2-B_1,A_3,A_4,\ldots, A_r),
	\end{align*}
	and invoke the following substitutions of variables:
	\begin{align*}
		r &\mapsto r-1,\\
		N &\mapsto N-B_1-1,\\
		X &\mapsto X-2B_1-2.
	\end{align*}
	Thus,
	\begin{align*}
		\sum_{B_2,\ldots,B_r} &= \big((N-B_1-1)!\big) \prod_{n=0}^{N-B_1-2}(X-N-B_1-1-n)\\
		&\quad\times \scalebox{0.8}{$\dfrac{1}{-B_1-1+\sum\limits_{j=1}^r (A_{r+1-j}+1)}$}\cdot \scalebox{0.8}{$\dfrac{1}{X-2N-B_1-1+\sum\limits_{j=1}^r (A_{r+1-j}+1)}$}\\
		&\quad\times \prod_{i=1}^{r-2} \scalebox{0.75}{$\dfrac{1}{\sum\limits_{j=1}^i (A_{r+1-j}+1)}$}  \prod_{i=1}^{r-2} \scalebox{0.75}{$\dfrac{1}{X-2N+\sum\limits_{j=1}^i (A_{r+1-j}+1)}$}.
	\end{align*}
	Now we have, with $N=\sum_{j=1}^r (A_j+1)$ in mind, that
	\begin{align*}
		\LHS\eqref{eq:aux-sum-1} &= \prod_{i=1}^{r-2} \scalebox{0.75}{$\dfrac{1}{\sum\limits_{j=1}^i (A_{r+1-j}+1)}$} \prod_{i=1}^{r-2} \scalebox{0.75}{$\dfrac{1}{X-2N+\sum\limits_{j=1}^i (A_{r+1-j}+1)}$}\\
		&\quad\times \sum_{B_1\ge 0} \big((N-B_1-2)!\big) \big(B_1!\big)^2 \binom{A_1}{B_1} \binom{X-2-A_1}{B_1}\\
		&\quad\times \prod_{n=N+B_1+2}^{2N-1}(X-n).
	\end{align*}
	It remains to prove that the above equals
	\begin{align*}
		\big(N!\big) \prod_{i=1}^r \scalebox{0.75}{$\dfrac{1}{\sum\limits_{j=1}^i (A_{r+1-j}+1)}$} \prod_{n=0}^{N-1}(X-N-n) \prod_{i=1}^r \scalebox{0.75}{$\dfrac{1}{X-2N+\sum\limits_{j=1}^i (A_{r+1-j}+1)}$},
	\end{align*}
	or equivalently,
	\begin{align}\label{eq:aux-sum-1-step}
		&\frac{(N-1)(X-N-1)}{(N-A_1-1)(X-N-A_1-1)}\\
		&\qquad\qquad= \sum_{B_1\ge 0} \big(B_1!\big)^2 \binom{A_1}{B_1} \binom{X-2-A_1}{B_1}\prod_{n=2}^{B_1}\frac{1}{(N-n)(X-N-n)}.\notag
	\end{align}
	We reformulate the right-hand side of the above in terms of the hypergeometric series and obtain
	\begin{align*}
		\RHS\eqref{eq:aux-sum-1-step} &= \sum_{B_1\ge 0} \frac{(-1)^{B_1} (-A_1)_{B_1} (-1)^{B_1} (-X+2+A_1)_{A_1}}{(-1)^{B_1} (-N+2)_{B_1} (-1)^{B_1} (-X+N+2)_{B_1}}\\
		&= {}_{3}F_{2}\left(\begin{matrix}
			-A_1,-X+2+A_1,1\\
			-X+N+2,-N+2
		\end{matrix};1\right)\\
		\text{\tiny (by \eqref{eq:P-S})}&= \frac{(N-A_1)_{A_1} (-X+N+1)_{A_1}}{(-X+N+2)_{A_1} (N-A_1-1)_{A_{1}}}\\
		&= \frac{(N-1)(-X+N+1)}{(-X+N+A_1+1)(N-A_1-1)},
	\end{align*}
	thereby arriving at the left-hand side of \eqref{eq:aux-sum-1-step}. This closes the required inductive argument.
\end{proof}

The next identity is of the same flavor but bears a more complicated sum side.

\begin{lemma}\label{combinatorial1}
	Let $\mathbf{A}=(A_{1},\ldots,A_{r})$ be a sequence (which may be empty) of nonnegative integers and let $M\ge r$ be an integer. Then for any indeterminates $X$ and $Y$,
	\begin{align}\label{eq:aux-sum-2}
		&\sum_{\mathbf{C}} \big(C_{r+1}!\big)^2 \binom{X+C_{r+1}-1}{C_{r+1}}\binom{-Y+M}{C_{r+1}}\prod_{i=1}^{r} (-1)^{C_i}\big(C_i!\big)^2\\
		&\times \scalebox{0.8}{%
			$\left(\begin{array}{c}
				\sum\limits_{j=1}^i A_j - \sum\limits_{j=1}^{i-1} C_j\\[12pt]
				C_i
			\end{array}\right)
			\left(\begin{array}{c}
				X+Y+M-2i-\sum\limits_{j=1}^i A_j - \sum\limits_{j=1}^{i-1} C_j\\[12pt]
				C_i
			\end{array}\right)$
		}\notag\\
		&\qquad\qquad\qquad\qquad\qquad\qquad = (-1)^{M-r} \prod_{n=0}^{M-1}(X+n) \prod_{i=1}^r \scalebox{0.65}{$\dfrac{1}{X+M-\sum\limits_{j=1}^i (A_{j}+1)}$}\notag\\
		&\qquad\qquad\qquad\qquad\qquad\qquad\quad\times \prod_{n=0}^{M-1}(Y-1-n) \prod_{i=1}^r \scalebox{0.65}{$\dfrac{1}{Y-\sum\limits_{j=1}^i (A_{j}+1)}$},\notag
	\end{align}
	where the summation runs over sequences $\mathbf{C}=(C_1,\ldots,C_{r+1})$ of nonnegative integers such that
	\begin{align*}
		\sum_{i=1}^{r+1} C_i = M-r.
	\end{align*}
\end{lemma}

\begin{proof}
	We prove this identity by induction on the length $r$ of the sequence $\mathbf{A}$. First, if $r=0$ so that $\mathbf{A}$ is the empty sequence, then the only choice of $\mathbf{C}$ is $\mathbf{C}=(M)$. Thus
	\begin{align*}
		\LHS\eqref{eq:aux-sum-2} &= \big(M!\big)^2 \binom{X+M-1}{M} \binom{-Y+M}{M}\\
		& = \prod_{n=0}^{M-1} (X+M-1-n) \prod_{n=0}^{M-1} (-Y+M-n),
	\end{align*}
	which is exactly the right-hand side of \eqref{eq:aux-sum-2}.
	
	Now let us assume that the identity is true for all sequences $\mathbf{A}$ of length $0,\ldots,r-1$ for a certain $r\ge 1$, and we want to prove the relation for an arbitrary sequence $\mathbf{A}=(A_{1},\ldots,A_{r})$. Our starting point is to single out the summation on $C_1$, so as to get
	\begin{align*}
		&\LHS\eqref{eq:aux-sum-2} \\
		&\quad= \sum_{C_1}(-1)^{C_1} \big(C_1!\big)^2 \binom{A_1}{C_1} \binom{X+Y+M-2-A_1}{C_1}\\
		&\quad\quad \times \sum_{C_2,\ldots,C_{r+1}}\big(C_{r+1}!\big)^2 \binom{X+C_{r+1}-1}{C_{r+1}}\binom{-Y+M}{C_{r+1}}\\
		&\quad\quad \times \prod_{i=2}^{r} (-1)^{C_i}\big(C_i!\big)^2 \scalebox{0.8}{%
			$\left(\begin{array}{c}
				\sum\limits_{j=1}^i A_j - \sum\limits_{j=1}^{i-1} C_j\\[12pt]
				C_i
			\end{array}\right)
			\left(\begin{array}{c}
				X+Y+M-2i-\sum\limits_{j=1}^i A_j - \sum\limits_{j=1}^{i-1} C_j\\[12pt]
				C_i
			\end{array}\right)$
		}.
	\end{align*}
	For the inner summation, we may apply the inductive assumption for the sequence of length $r-1$:
	\begin{align*}
		\widehat{\mathbf{A}} = (A_1+A_2-C_1,A_3,A_4,\ldots, A_r).
	\end{align*}
	Furthermore, we make the following change of variables:
	\begin{align*}
		r &\mapsto r-1,\\
		M &\mapsto M-C_1-1,\\
		Y &\mapsto Y-C_1-1.
	\end{align*}
	Therefore,
	\begin{align*}
		\sum_{C_2,\ldots,C_{r+1}} &= (-1)^{M-r-C_1} \prod_{n=0}^{M-C_1-2}(X+n) \prod_{i=2}^r \scalebox{0.65}{$\dfrac{1}{X+M-\sum\limits_{j=1}^i (A_{j}+1)}$}\\
		&\quad\times \prod_{n=0}^{M-C_1-2}(Y-C_1-2-n) \prod_{i=2}^r \scalebox{0.65}{$\dfrac{1}{Y-\sum\limits_{j=1}^i (A_{j}+1)}$},
	\end{align*}
	which further gives us
	\begin{align*}
		\LHS\eqref{eq:aux-sum-2} &= (-1)^{M-r}\prod_{i=2}^r \scalebox{0.65}{$\dfrac{1}{X+M-\sum\limits_{j=1}^i (A_{j}+1)}$}\prod_{i=2}^r \scalebox{0.65}{$\dfrac{1}{Y-\sum\limits_{j=1}^i (A_{j}+1)}$}\\
		&\quad\times \sum_{C_1} \big(C_1!\big)^2 \binom{A_1}{C_1} \binom{X+Y+M-2-A_1}{C_1}\\
		&\quad\times\prod_{n=0}^{M-C_1-2}(X+n)\prod_{n=0}^{M-C_1-2}(Y-C_1-2-n).
	\end{align*}
	Now we are left to show that the above equals
	\begin{align*}
		(-1)^{M-r} \prod_{n=0}^{M-1}(X+n) \prod_{i=1}^r \scalebox{0.65}{$\dfrac{1}{X+M-\sum\limits_{j=1}^i (A_{j}+1)}$} \prod_{n=0}^{M-1}(Y-1-n) \prod_{i=1}^r \scalebox{0.65}{$\dfrac{1}{Y-\sum\limits_{j=1}^i (A_{j}+1)}$},
	\end{align*}
	or equivalently,
	\begin{align}\label{eq:aux-sum-2-step}
		&\frac{(X+M-1)(Y-1)}{(X+M-A_1-1)(Y-A_1-1)}\\
		&\qquad\qquad=\sum_{C_1\ge 0} \big(C_1!\big)^2 \binom{A_1}{C_1} \binom{X+Y+M-2-A_1}{C_1}\notag\\
		&\qquad\qquad\quad\times\prod_{n=M-C_1-1}^{M-2}(X+n)\prod_{n=2}^{B_1+1}(Y-n). \notag
	\end{align}
	Note that
	\begin{align*}
		\RHS\eqref{eq:aux-sum-2-step} &= \sum_{C_1\ge 0} \frac{(-1)^{C_1} (-A_1)_{C_1} (-1)^{C_1} (-X-Y-M+2+A_1)_{C_1}}{(-1)^{C_1} (-X-M+2)_{C_1} (-1)^{C_1} (-Y+2)_{C_1}}\\
		&= {}_{3}F_{2}\left(\begin{matrix}
			-A_1,-X-Y-M+2+A_1,1\\
			-Y+2,-X-M+2
		\end{matrix};1\right)\\
		\text{\tiny (by \eqref{eq:P-S})}&= \frac{(X+M-A_1)_{A_1} (-Y+1)_{A_1}}{(-Y+2)_{A_1} (X+M-A_1-1)_{A_{1}}}\\
		&= \frac{(X+M-1)(-Y+1)}{(X+M-A_1-1)(-Y+A_1+1)}.
	\end{align*}
	The above is the same as the left-hand side of \eqref{eq:aux-sum-2-step}, thereby completing the proof.
\end{proof}
\subsection*{Acknowledgements}

The authors would like to thank Jeffrey Case for useful discussions. ZY is supported by a Simons Travel Grant.

\bibliography{bib}

\end{document}